\documentclass[11pt,reqno]{amsart}

\usepackage[margin=1in]{geometry}

\title[Nonlinear maximum principles for dissipative linear nonlocal operators and applications]{Nonlinear maximum principles for dissipative linear nonlocal operators and applications}
\date{\today}
\author{Peter Constantin}
\address{Department of Mathematics,
The University of Chicago, 5734 University Ave.,
Chicago, IL 60637} \email{\tt const@cs.uchicago.edu}

\author{Vlad Vicol}
\address{Department of Mathematics,
The University of Chicago, 5734 University Ave.,
Chicago, IL 60637} \email{\tt vicol@math.uchicago.edu}

\usepackage{amsfonts,amsmath,latexsym,amssymb,verbatim,amsbsy,times,color,mathrsfs,amsthm}

\usepackage{hyperref}

\theoremstyle{plain}
\newtheorem{theorem}{Theorem}[section]
\newtheorem{definition}[theorem]{Definition}

\newtheorem{proposition}[theorem]{Proposition}
\newtheorem{corollary}[theorem]{Corollary}

\theoremstyle{definition}
\newtheorem{remark}[theorem]{Remark}

\def\tilde{\widetilde}

\numberwithin{equation}{section}

\renewcommand\hat{\widehat}
\def\ZZ{{\mathbb Z}}

\def\RR{{\mathbb R}}
\def\TT{{\mathbb T}}
\def\Sphere{{\mathbb S}}
\def\Schwartz{{\mathscr S}}
\def\RSZ{{\mathcal R}}
\def\LL{{\mathcal L}}

\newcommand{\BB}[2]{(B$_{#1,#2}$)}
\newcommand{\fr}{\frac}
\newcommand{\pa}{\partial}
\newcommand{\bn}{\begin{align}}
\newcommand{\en}{\end{align}}
\newcommand{\la}{\label}

\begin{document}

%%%%%%%%%%%%%%%%%%%%%%%%% THE ABSTRACT %%%%%%%%%%%%%%%%%%%%%%%%%%%%%%%%%%%

\begin{abstract} We obtain a family of nonlinear maximum principles for linear dissipative nonlocal operators, that are general, robust, and versatile. We use these nonlinear bounds to provide transparent proofs of global regularity for critical SQG and critical d-dimensional Burgers equations. In addition we give applications of the nonlinear maximum principle to the global regularity of a slightly dissipative anti-symmetric perturbation of 2d incompressible Euler equations and generalized fractional dissipative 2d Boussinesq equations.    
\end{abstract}

%%%%%%%%%%%%%%%%%%%%%%%% Classification and Keywords %%%%%%%%%%%%%%%%%%%%


\maketitle
%%%%%%%%%%%%%%%%%%%%%%%%%% The Main Part %%%%%%%%%%%%%%%%%%%%%%%%%%%%%%%%%%

\section{Introduction}\label{sec:intro}
How can a {\em linear} operator obey  a {\em nonlinear} maximum principle?  We are referring to shape-dependent bounds, of the type
\begin{align}
(\Lambda^{\alpha} g)(\bar{x}) \ge \fr{g(\bar{x})^{1+\alpha}}{c\, m^{\alpha}} \label{eq:intro}
\end{align}
where $\Lambda = \sqrt{-\Delta}$, $g =\partial f$ is a scalar function, the directional derivative $\pa$ of some other scalar function $f$, $\bar{x}$ is a point in $\RR^d$ where $g$ attains its maximum, $m= \|f\|_{L^{\infty}}$, $c>0$ is a constant, and $0<\alpha<2$. In fact, the bound \eqref{eq:intro} scales linearly with $f$  and correctly with respect to dilations (as it should) but it has a nonlinear dependence on the maximum of $g$. We refer to such inequalities as nonlinear maximum principles. They are  true for fractional powers of the Laplacian and for many other nonlocal dissipative operators, and they are versatile and robust.  

When studying nonlinear evolution equations, we often encounter  situations in which the equation has some conserved quantities, but these a priori controlled quantities are not strong enough to guarantee global existence of smooth solutions. In fact smooth solutions may break down. The basic example of the Burgers equation
$$
\theta_t + \theta\theta_x =0
$$
is worth keeping in mind. We take $x\in\RR$. The norms $\|\theta\|_{L^p}$ are conserved under smooth evolution, $1\le p\le\infty$. Taking a derivative $g=\theta_x$
we have
$$
g_t +\theta g_x +g^2=0.
$$
This blows up in finite time. If one adds dissipation,
$$
\theta_t + \Lambda^{\alpha}\theta  + \theta\theta_x =0
$$
then $g$ obeys
$$
g_t + \Lambda^{\alpha} g + \theta g_x +g^2=0.
$$ 
This still blows up for $\alpha<1$, and does not blow up for $\alpha\ge 1$  \cite{KNS}. The reader can sense already how easily the regularity result would follow from the nonlinear maximum principle in the subcritical  $\alpha>1$  case. At the critical exponent $\alpha =1$, the nonlinear maximum principle readily proves the global regularity of solutions with small $L^{\infty}$ norms. In order to remove this restriction one has to recognize an additional structure in the equation: the stability of small shocks. This is discussed in further detail in Section~\ref{sec:Burgers} below.

A similar situation is encountered in the study of the dissipative SQG equation.
\begin{align}
\partial_t \theta + u \cdot \nabla \theta +\Lambda^\alpha \theta = 0, \label{eq:motivation:1}
\end{align}
with divergence-free velocity $u$, related to $\theta$ by a constitutive law that puts $u$ on par with $\theta$. Here and throughout this paper we denote the Zygmund operator {by} $\Lambda = (-\Delta)^{1/2}$. The $L^p$ norms of $\theta$ are non-increasing in time under smooth evolution, and it is known that smooth solutions persist for $\alpha\ge 1$ (\cite{CV}, \cite{CW0}, \cite{KN1}, \cite{KN2}, \cite{KNV}, \cite{R}, and many more). The case $\alpha=1$ is universally termed  ``critical'', although it is not known yet if a critical change in behavior actually does occur at $\alpha =1$ (this occurs for Burgers, and so the name is well justified there). We have many analogous situations in PDE of hydrodynamic origin. The ``critical'' cases, are cases in which easy proofs break down, and when regularity indeed persists, the proofs are usually ingenious, involved and implicit. In the case of SQG, there are two quite different main proof ideas. The approach of \cite{KNV} is to find a modulus of continuity that is invariant in time. The interplay between nonlocal dissipation and nonlinearity is used in a subtle and very original way. The proof of \cite{CV} follows a strategy that has been associated to DeGiorgi: the existence of an a priori integral bound and dilation invariance are exploited by zooming in to small scales. In that proof the crucial step is a passage from $L^{\infty}$ information to $C^{\alpha}$ information. An alternative proof of $C^{\alpha}$ regularity has been recently obtained in \cite{KN2}, by a duality method, exploiting the co-evolution of molecules. 

In this paper we provide a new, transparent proof of global regularity for critical SQG based on the nonlinear maximum principle. The proof has two parts. The first part shows that if a bounded solution has {\em only small shocks} (OSS), a technical term that we define precisely in \eqref{def:oss} below, then it is a smooth solution. In the second part we show that if a solution has only small shocks to start with, then it does have only small shocks for all time. Both parts are proved using appropriate nonlinear maximum principles. We exemplify the same strategy for the Burgers equation, in any spatial dimension. The first part essentially shows how having OSS is a way of assuring that the dissipation beats the nonlinearity. In order to prove the second part we follow the structure of the equation that gave the conservation of $L^{\infty}$ norm, but we do it for displacements. Localizing to small displacements and requiring that
the resulting equation has a weak maximum principle leads to a localizer family, obeying a universal differential inequality, and it is the nature of this inequality that determines whether or not the persistence of the OSS condition takes place or not.

Nonlinear maximum principles are not relegated to fractional powers of the Laplacian. We give examples of other operators that have a nonlinear maximum principle in the study of an anti-symmetric, nonlocal perturbation of the the Euler equations. In fact, we prove that in the presence of an {\em arbitrarily weak} nonlocal dissipation, the anti-symmetric perturbation of the 2D Euler equation is globally regular. 

We also show global regularity for  mixed fractionally-dissipative 2D Boussinesq equations, under a certain condition on the powers of the fluid and the temperature dissipation.

Throughout this paper we make the convention that $c_1, c_2, \ldots$ denote positive universal constants, which may depend on the dimension of the space or on other parameters of the equation. On the other hand,  we shall denote by $C_1,C_2,\ldots$ constants which may depend on certain $L^p$ norms of the initial data.

\section{Nonlinear Maximum Principles} \label{sec:max}

We recall that the fractional power of the (negative) Laplacian, which may be defined via the Fourier transform as
\begin{align*}
\left( \Lambda^\alpha g\right)^{\hat{\ }} (\xi) = |\xi|^{\alpha} \hat{g}(\xi)
\end{align*}
is given in real variables, when  $0<\alpha<2$, as the principal value of the integral
\begin{align*}
\Lambda^\alpha g(x) &=  c_{d,\alpha}\; P.V. \int_{\RR^{d}} \frac{g( x) - g(x-y)}{|y|^{d+\alpha}}\; dy
\end{align*}
where $c_{d,\alpha} = \pi^{-(\alpha+d/2)} \Gamma(\alpha/2 + d/2) \Gamma(-\alpha/2)^{-1}$ is a normalizing constant, which degenerates as $\alpha \rightarrow 2$ and as $\alpha \rightarrow 0$. Although in this paper we may sometimes omit the $P.V.$ in front of the integral defining $\Lambda^\alpha$, the integral is always understood in the principal value sense.

The main result of this section is the following {\em nonlinear lower bound} on the fractional Laplacian, evaluated at the maximum of a smooth function. 
\begin{theorem} [\bf $L^\infty$ nonlinear lower bound]\label{thm:lowerbound}
Let $f \in \Schwartz(\RR^{d})$. For a fixed $k\in \{1,\ldots,d\}$, let $g(x) = \partial_{k} f(x)$. Assume that $\bar x \in \RR^{d}$ such that $g(\bar x) = \max_{x \in \RR^{d}} g(x) >0$. Then  we have
\begin{align}
\Lambda^\alpha g (\bar x) \geq \frac{ g(\bar x)^{1+\alpha} }{c \Vert f \Vert_{L^{\infty}}^{\alpha}} \label{eq:lowerbound}
\end{align}
for $\alpha \in (0,2)$, and some universal positive constant $c=c(d,\alpha)$ which may be computed explicitly. %and blows up as $1/(2-\alpha)$ when $\alpha \rightarrow 2$.
\end{theorem}

Nonlinear lower bounds appeared in the recent work of Kiselev and Nazarov~\cite{KN2}, where they use a nonlinear lower bound to estimate the fractional Laplacian evaluated at the maximum of a function, in terms of the $L^1$ or $L^2$ norm of the function. The main difference here is that we go forward one derivative in regularity: nonlinear information on $\nabla f$ is obtained from information of $f$.

\begin{proof}[Proof of Theorem~\ref{thm:lowerbound}]
Let $R > 0$ be fixed, to be chosen later, and let $\chi$ be a radially non-decreasing smooth cut-off function, which vanishes on $|x|\leq 1$ and is identically $1$ on $|x|\geq 2$, and $|\nabla \chi| \leq 4$. We have
\begin{align}
\Lambda^\alpha g(\bar x) 
&= c_{d,\alpha} \int_{\RR^{d}} \frac{g(\bar x) - g(\bar x-y)}{|y|^{d+\alpha}}\; dy \notag\\
&\geq c_{d,\alpha} \int_{\RR^{d}} \frac{g(\bar x) - g(\bar x-y)}{|y|^{d+\alpha}} \chi(y/R) \;dy \notag\\
&\geq c_{d,\alpha} g(\bar x) \int_{|y|\geq 2R} \frac{dy}{|y|^{d+\alpha}} - c_{d,\alpha}\int_{\RR^{d}} |f(\bar x - y)| \left| \partial_{y_k} \frac{\chi(y/R)}{|y|^{d+\alpha}} \right| \; dy \label{eq:lower:g:1}
\end{align}
where in the last inequality we integrated by parts since $g = \partial_{k} f$, and $c_{d,\alpha}$ is the normalization constant of the fractional Laplacian. After a short calculation it follows from \eqref{eq:lower:g:1} that
\begin{align}
\Lambda^\alpha g(\bar x) &\geq c_1 \frac{g(\bar x)}{R^{\alpha}} - c_2 \frac{\Vert f \Vert_{L^{\infty}}}{R^{\alpha +1}}\label{eq:lower:g:2}
\end{align}
where $c_1 = |\Sphere^{d-1}| c_{d,\alpha} 2^{-\alpha} \alpha^{-1}$, and $c_2 = |\Sphere^{d-1}|c_{d,\alpha} (4 d + \alpha) \alpha^{-1}$. Inserting
\begin{align*}
R = \frac{2 c_2 \Vert f \Vert_{L^{\infty}}}{c_1 g(\bar x)}
\end{align*}
into estimate \eqref{eq:lower:g:2} concludes the proof of the theorem. Moreover, the constant $c = c(d,\alpha)$ in \eqref{eq:lowerbound} may be taken explicitly to be $\alpha 2^{(1+\alpha)^2} (4+d)^\alpha |\Sphere^{d-1}|^{-1} c_{d,\alpha}^{-1}$. 
%Note that since $C_{d,\alpha}$ behaves like $2-\alpha$ when $\alpha$ is close to $2$, we do not obtain the corresponding lower bound when $\alpha = 2$. The bound does however (trivially) hold for $\alpha = 0$, since $C_{d,\alpha} = {\mathcal O}(\alpha)$, when $\alpha \rightarrow 0$.
\end{proof}

%\begin{remark}
%  As noted above, in the limit $\alpha \rightarrow 2$, the lower bound \eqref{eq:lowerbound} becomes trivial, since the constant $C$ blows up as $1/(2-\alpha)$. In fact one may even construct explicit examples of non-trivial functions $f,g$ such that $g(\bar x) = \max g(x) > 0$, but with $\Delta g (\bar x) = 0$, showing that in the case $\alpha=2$ the lower bound \eqref{eq:lowerbound} cannot hold in general.
%\end{remark}

If we have more a priori information on $f$, such as a bound in $C^\delta(\RR^d)$ for some $\delta>0$, or respectively less information of $f$, such as $f \in L^p(\RR^d)$ for some $p \geq 1$, then the following results complement Theorem~\ref{thm:lowerbound}:
\begin{theorem}[\bf $C^\delta$ and $L^p$ nonlinear lower bound] \label{thm:lowerbound:extended}
  Let $f \in \Schwartz(\RR^d)$, $g = \partial_k f$ for some $k \in \{1,\ldots,d\}$, and let $\bar x \in \RR^d$  be such that $g(\bar x) = \max_{x\in \RR^d} g(x)$. Then we have
  \begin{align}
    \Lambda^\alpha g(\bar x) \geq \frac{ g(\bar x)^{1+ \frac{\alpha}{1-\delta}}}{ c \Vert f \Vert_{C^\delta}^{\frac{\alpha}{1-\delta}}}\label{eq:Holder:lowerbound}
  \end{align}
  for $\delta \in (0,1)$, and a positive constant $c=c(d,\alpha,\delta)$. In addition, the bound
  \begin{align}
   \Lambda^\alpha g(\bar x) \geq  \frac{ g(\bar x)^{1+ \frac{\alpha p}{d+p}}}{ c \Vert f \Vert_{L^p}^{\frac{\alpha p}{d+p}}} \label{eq:Lp:lowerbound}
  \end{align}
  holds for any $p \geq 1$, for some constant $c=c(d,\alpha,p)$.
\end{theorem}
\begin{proof}[Proof of Theorem~\ref{thm:lowerbound:extended}]
Let $R>0$, to be chosen later, and let $\chi$ be the same smooth cut-off function from the proof of Theorem~\ref{thm:lowerbound}.
In order to prove \eqref{eq:Holder:lowerbound}, similarly to \eqref{eq:lower:g:1}, we estimate
\begin{align*}
\Lambda^\alpha g(\bar x) &= c_{d,\alpha} \int_{\RR^{d}} \frac{g(\bar x) - g(x-y)}{|y|^{d+\alpha}}\; dy \notag\\
&\geq c_{d,\alpha} g(\bar x) \int_{\RR^d} \frac{\chi(y/R)}{|y|^{d+\alpha}} dy - c_{d,\alpha} \left| \int_{\RR^d} \partial_{y_k} \left(f(\bar x) - f(\bar x - y) \right) \frac{\chi(y/R)}{|y|^{d+\alpha}} dy \right|\notag\\
&\geq c_{d,\alpha} g(\bar x)\int_{|y|\geq 2R} \frac{1}{|y|^{d+\alpha}} dy - c_{d,\alpha} \Vert f \Vert_{C^\delta} \int_{\RR^d} |y|^\delta \left| \partial_{y_k}  \frac{\chi(y/R)}{|y|^{d+\alpha}} \right| dy\notag\\
&\geq c_1  \frac{g(\bar x)}{R^\alpha}  - c_2 \frac{\Vert f \Vert_{C^\delta} }{R^{\alpha+1-\delta}}
\end{align*}
where $c_1 = c_1(d,\alpha)$, and $c_2 = c_2(d,\alpha,\delta)$ are positive constants, which may be computed explicitly. In the above estimate we have used that $\partial_{y_k} f(\bar x)= 0$, and have integrated by parts in $y_k$. Letting $R^{1-\delta} =2 c_2 \Vert f \Vert_{C^\delta}/ (c_1 g(\bar x))$ concludes the proof of the lower bound \eqref{eq:Holder:lowerbound}.

In order to prove the corresponding lower bound in terms of $\Vert f \Vert_{L^p}$, we use \eqref{eq:lower:g:1} and the H\"older inequality
\begin{align*}
 \Lambda^\alpha g(\bar x) & \geq  c_{d,\alpha} g(\bar x) \int_{|y|\geq 2R} \frac{dy}{|y|^{d+\alpha}} - c_{d,\alpha}\int_{\RR^{d}} |f(\bar x - y)| \left| \partial_{y_k} \frac{\chi(y/R)}{|y|^{d+\alpha}} \right| \; dy \notag\\
  &\geq c_1 \frac{g(\bar x)}{R^\alpha} - c_2 \frac{\Vert f \Vert_{L^p}}{R^{\alpha+1+ \frac dp}}
\end{align*}
where $c_1 = c_1 (d,\alpha)$ and $c_2 = c_2(d,\alpha,p)$ are explicitly computable positive constants. Letting $R^{1+ d/p} = 2 c_2 \Vert f \Vert_{L^p} / (c_1 g(\bar x))$ concludes the proof of \eqref{eq:Lp:lowerbound} and of the lemma.
\end{proof}

A similar nonlinear lower bound to \eqref{eq:lowerbound} may also be obtained if the function is  $\TT^{d} = [-\pi,\pi]^{d}$ periodic.
\begin{theorem}[\bf $L^\infty$ periodic lower bound]\label{thm:periodic}
Let $f \in C^{\infty}(\TT^{d})$ and $g = \partial_{k} f$. Given $\bar x \in \TT^{d}$ such that $g(\bar x) = \max_{x\in\TT^{d}} g(x)$, there exits a positive constant $c$ such that either
\begin{align}
g(\bar x) \leq c \Vert f \Vert_{L^{\infty}} \label{eq:periodic:lower:1}
\end{align}
or
\begin{align}
\Lambda^\alpha g(\bar x) \geq \frac{ g(\bar x)^{1+\alpha}}{c \Vert f \Vert_{L^{\infty}}^{\alpha}} \label{eq:periodic:lower:2}
\end{align}
holds.
\end{theorem}
\begin{proof}[Proof of Theorem~\ref{thm:periodic}]
Recall (cf.~\cite{CC} for instance) that in the periodic setting the integral expression (in principal value) for the fractional Laplacian is
\begin{align*}
\Lambda^\alpha g(\bar x) = c_{d,\alpha} \sum_{j \in \ZZ^{d}} \int_{\TT^{d}} \frac{g(\bar x) - g(\bar x + y)}{|y+j|^{d+\alpha}}\; dy.
\end{align*}
The main contribution to the sum comes from the term with $j = 0$, since only then the kernel is singular. Also, since $\bar x$ is the point of maximum of $g$, all other terms are positive, so be have the lower bound
\begin{align*}
\Lambda^\alpha g(\bar x) &\geq c_{d,\alpha} \int_{\TT^{d}} \frac{g(\bar x) - g(\bar x + y)}{|y|^{d+\alpha}} \chi(y/R)\; dy
\end{align*}
where $\chi$ is a smooth cut-off as in the proof of Theorem~\ref{thm:lowerbound}, and $R>0$ is to be chosen later. As before, we obtain
\begin{align*}
\Lambda^\alpha g(\bar x) &\geq c_{1} g(\bar x) \left( \frac{1}{(2R)^{\alpha}} - \frac{1}{\pi^{\alpha}} \right) - c_{2}\frac{ \Vert f \Vert_{L^{\infty}} }{R^{1+\alpha}}\notag\\
&\geq \frac{c_{1}}{2^{1+\alpha}}  \frac{g(\bar x)}{R^{\alpha}} - c_{2} \frac{ \Vert f \Vert_{L^{\infty}} }{R^{1+\alpha}}
\end{align*}
by requiring that $R< \pi/2^{1+1/\alpha}$. We would like to let $R = c_{2} 2^{2+\alpha} \Vert f \Vert_{L^{\infty}}/(c_{1} g(\bar x))$, in order to obtain
\begin{align*}
\Lambda^\alpha g(\bar x) \geq \frac{ (g(\bar x))^{1+\alpha}}{c_{3} \Vert f \Vert_{L^{\infty}}^{\alpha}},
\end{align*}
where $c_{3} = c_{1}^{1+\alpha}/(2^{1 + 3\alpha + \alpha^{2}}c_{2}^{\alpha})$, but this is only possible, due to the restriction the size of $R$, if
\begin{align}
g(\bar x) \geq c_{4} \Vert f \Vert_{L^{\infty}} \label{eq:periodic:technical}
\end{align}
holds, where $c_{4} = c_{2} 2^{3 + \alpha + 1/\alpha}/(\pi c_{1})$. Letting $c= c_{3} + c_{4}$ concludes the proof of the theorem.
\end{proof}
\begin{remark}
It is clear from the proof of Theorem~\ref{thm:periodic}, that the analogue of Theorem~\ref{thm:lowerbound:extended} also holds in the periodic setting. Namely, either $g(\bar x)$ can be controlled by a multiple of $\Vert f \Vert_{C^{\delta}}$ (respectively $\Vert  f \Vert_{L^{p}}$), or \eqref{eq:Holder:lowerbound} (respectively \eqref{eq:Lp:lowerbound}) holds.
\end{remark}

Lastly, we note that a nonlinear lower bound in the spirit of Theorem~\ref{thm:lowerbound} also holds for the positive scalar quantity $|\nabla f|^2$. This bound turns out to be very useful in applications. More precisely, we have:
\begin{theorem}[\bf Pointwise nonlinear lower bound]\label{thm:lowerbound:scalar}
Let $f \in \Schwartz(\RR^{d})$. Then  we have the pointwise bound
\begin{align}
\nabla f  (x) \cdot \Lambda^\alpha \nabla f (x) \geq \frac{1}{2} \Lambda^\alpha |\nabla f(x)|^2 + \frac{ |\nabla f (x)|^{2+\alpha} }{c \Vert f \Vert_{L^{\infty}}^{\alpha}} \label{eq:lowerbound2}
\end{align}
for $\alpha \in (0,2)$, and some universal positive constant $c=c(d,\alpha)$.
\end{theorem}
\begin{proof}[Proof of Theorem~\ref{thm:lowerbound:scalar}]
We use the pointwise identity 
\begin{align}
\nabla f(x)  \cdot \Lambda^\alpha \nabla f (x) 
= \frac{1}{2}\Lambda^\alpha ( |\nabla f|^2) (x) 
+ \frac{1}{2} D
\label{eq:pointwise}
\end{align}
where we have denoted (in principal value sense)
\begin{align}
D = c_{d,\alpha} \int_{\RR^d}  \frac{ |\nabla f(x) - \nabla f (x+y)|^2}{|y|^{d+\alpha}} dy. \label{eq:DDD:def}
\end{align}
The pointwise identity \eqref{eq:pointwise} follows from the argument in~\cite{CC} (see also~\cite{C1}). Here $c_{d,\alpha}$ is the normalizing constant of the integral expression of the fractional Laplacian. We now bound from below
\begin{align}
D 
& \geq c_{d,\alpha} \int_{\RR^d}  \frac{ |\nabla f(x) - \nabla f (x+y)|^2}{|y|^{d+\alpha}} \chi(y/R) dy 
\label{eq:gradient:lower:1}
\end{align}
where as before $\chi$ is a smooth radially non-increasing cut-off function that vanishes on $|x|\leq 1/2$ and is identically $1$ on $|x|\geq 1$. For all $y$ we have 
$$
|\nabla f(\bar x) - \nabla f(\bar x + y)|^2 \geq |\nabla f (\bar x)|^2 - 2 \partial_j f(\bar x) \partial_j f(\bar x + y),
$$ 
where the summation convention on repeated indices is used. Hence, from \eqref{eq:gradient:lower:1} it follows that
\begin{align}
D& \geq c_{d,\alpha} |\nabla f(\bar x)|^2 \int_{\RR^d} \frac{ \chi(y/R) }{|y|^{d+\alpha}} dy - c_{d,\alpha} |\partial_j f(\bar x)| \left| \int_{\RR^2} \partial_j f(\bar x + y) \frac{\chi(y/R)}{|y|^{d+\alpha}} dy \right| \notag\\
&\geq c_{d,\alpha} |\nabla f(\bar x)|^2 \int_{|y|\geq R} \frac{1}{|y|^{d+\alpha}} dy   - c_{d,\alpha} |\partial_j f(\bar x)| \Vert f \Vert_{L^\infty}  \int_{\RR^2} \left| \partial_j \frac{\chi(y/R)}{|y|^{d+\alpha}} \right| dy\notag\\
&\geq c_1 \frac{|\nabla f (\bar x)|^2}{R^\alpha} - c_2 \frac{|\nabla f(\bar x)|\; \Vert f \Vert_{L^\infty}}{R^{1+\alpha}}
\label{eq:gradient:lower:2}
\end{align}
for some positive constants $c_1$ and $c_2$ which depend only on $d, \alpha$, and $\chi$.
Letting $$R = \frac{c_2 \Vert f \Vert_{L^\infty}}{2 c_1 |\nabla f (\bar x)|}$$ concludes the proof of the Theorem.
\end{proof}
\begin{corollary}\label{cor:lowerbound:scalar}
Let $f \in \Schwartz(\RR^{d})$. Assume that there exists $\bar x \in \RR^{d}$ such that the maximum of  $|\nabla f (x) |^2$ is attained at $\bar x$. Then  we have
\begin{align*}
\nabla f (\bar x) \cdot \Lambda^\alpha \nabla f (\bar x) \geq \frac{ |\nabla f (\bar x)|^{2+\alpha} }{c \Vert f \Vert_{L^{\infty}}^{\alpha}}
\end{align*}
for $\alpha \in (0,2)$, and some universal positive constant $c=c(d,\alpha)$.
\end{corollary}
\begin{proof}[Proof of Corollary~\ref{cor:lowerbound:scalar}]
This follows from Theorem~\ref{thm:lowerbound:scalar} by noting that at the maximum of $|\nabla f|^2$, which is by assumption attained at $\bar x$, the term $\Lambda^\alpha (|\nabla f|^2)(\bar x)$ is non-negative. 
\end{proof}

\section{Applications to the dissipative SQG equations} \label{sec:SQG}

The dissipative surface quasi-geostrophic equation (SQG)
\begin{align}
  &\partial_t \theta + u \cdot \nabla \theta + \Lambda^\alpha \theta = 0 \label{eq:SQG:1}\\
  & u = \RSZ^\perp \theta \label{eq:SQG:2}\\
  & \theta(\cdot,0) = \theta_0 \label{eq:SQG:3}
\end{align}
has recently attracted a lot of attention in the mathematical literature, see for instance the extended list of references in \cite{CCW}. Here $0 < \alpha < 2$. While the global regularity in the sub-critical case $\alpha>1$ has been long ago established \cite{R}, \cite{CW0}, the global regularity in the critical case $\alpha=1$ has been proven only recently \cite{CV,KN2,KNV}. In the super-critical case $\alpha<1$, with large initial data, only eventual regularity~\cite{D,Silvestre} and conditional regularity~\cite{CW08,CW09} have been established.

In this section we establish give {\em a new proof} of the global well-posedness of \eqref{eq:SQG:1}--\eqref{eq:SQG:2} in the critical case $\alpha = 1$. The proof is based on the nonlinear maximum principle established earlier in section~\ref{sec:max}, and is split in two steps. The first step shows that if a solution of the SQG equation has ``only small  shocks'', then it is regular (cf.~Theorem~\ref{thm:step1} below), while the second step shows that if the initial data has only small shocks, then the solution has only small shocks for all later times (cf.~Theorem~\ref{thm:step2} below). To be more precise let us introduce:

\begin{definition}[{\bf Only Small Shocks}]\label{def:oss}
Let $\delta>0$, and $t>0$. We say  $\theta(x,t)$ has the {\em $OSS_\delta$ property}, if there exists an $L>0$ such that
\begin{align}
\sup_{ \{ (x,y) \colon |x-y|<L \}} |\theta(x,t) - \theta(y,t) | \leq \delta. \label{eq:Pdelta}
\end{align}
Moreover, for $T>0$, we say  $\theta(x,t)$ has the {\em uniform $OSS_{\delta}$ property} on $[0,T]$, if there exists an $L>0$ such that
\begin{align}
\sup_{ \{ (x,y,t) \colon |x-y|<L,\; 0\leq t \leq T \}} |\theta(x,t) - \theta(y,t) | \leq \delta. \label{eq:PdeltaT}
\end{align}
\end{definition}
Our first result states that the uniform $OSS_\delta$ property implies regularity of the solution:
\begin{theorem}[\bf From Only Small Shocks to regularity] \label{thm:step1}
There exists a $\delta_0>0$, depending only on $\| \theta_0 \|_{L^\infty}$, so that if $\theta$ is a bounded weak solution of the critical SQG equation with the uniform $OSS_{\delta_0}$ property on $[0,T]$, then it is a smooth solution on $[0,T]$. Moreover 
\begin{align}
\sup_{t\in[0,T]} \| \nabla \theta(\cdot,t) \|_{L^\infty} \leq C(\|\theta_0\|_{L^\infty},\|\nabla \theta_0\|_{L^\infty},L)
\label{eq:SQG:Gradient:Bound}
\end{align} 
where $L$ is defined as in \eqref{eq:PdeltaT}.
\end{theorem}
In fact we give the proof of \eqref{eq:SQG:Gradient:Bound} assuming $\theta$ is a smooth function on $[0,T)$. These arguments can then be then made formal by adding a hyper-regularization $-\varepsilon \Delta \theta$ to the equations. Since the proof given below carries through to the regularized equations, and the bounds obtained are $\varepsilon$-independent, we may pass to a limit as $\varepsilon \to 0$, and obtain \eqref{eq:SQG:Gradient:Bound} for the SQG equations.
Any subcritical regularization would do. The main point is that the estimate
\ref{eq:SQG:Gradient:Bound} is uniform in time. 
\begin{proof}[Proof of Theorem~\ref{thm:step1}]
It is clear from \eqref{eq:SQG:1} that $\Vert \theta(\cdot,t) \Vert_{L^\infty} \leq \Vert \theta_0 \Vert_{L^\infty}$, and hence a suitable a priori estimate on $\Vert \nabla \theta \Vert_{L^\infty}$ implies that $\theta$ is in fact a smooth solution. For this purpose, apply $\nabla$ to \eqref{eq:SQG:1} and multiply by $\nabla\theta$ to obtain
\begin{align}
\frac{1}{2} (\partial_t + u \cdot \nabla) |\nabla \theta|^2 + \nabla \theta \cdot \Lambda \nabla \theta + \nabla u \colon \nabla \theta \cdot \nabla \theta = 0, \label{eq:pointwise:evolution}
\end{align}
where as usual we denote $\Lambda = (-\Delta)^{1/2}$. Recall cf.~\eqref{eq:lowerbound2} that we have the pointwise identity
\begin{align}
\nabla \theta(x) \cdot \Lambda \nabla \theta(x) = \frac{1}{2} \Lambda |\nabla \theta(x)|^2  + \frac{1}{2} D(x,t) \label{eq:pointwise:D}
\end{align}
where
\begin{align*}
D(x,t) = c_0 P.V. \int_{\RR^d} \frac{ |\nabla \theta(x,t) - \nabla \theta(y,t)|^2}{|x-y|^{d+1}} dy,
\end{align*}
with $c_0 = c_{d,1}$ being the normalizing constant of $\Lambda$. As shown in the proof of Theorem~\ref{thm:lowerbound:scalar}, there exists a non-dimensional constant $c_1>0$ such that
\begin{align}
\frac{1}{4} D(x,t) \geq c_1 \frac{ |\nabla \theta(x,t)|^3}{\| \theta_0 \|_{L^\infty}}. \label{eq:pointwise:D:lower}
\end{align}
Here we also used the $L^\infty$ maximum principle for $\theta$. Summarizing, \eqref{eq:pointwise:evolution}, \eqref{eq:pointwise:D}, and \eqref{eq:pointwise:D:lower} give
\begin{align}
\frac{1}{2} \left( \partial_t + u \cdot \nabla + \Lambda \right) |\nabla \theta(x,t)|^2 + c_1 \frac{ |\nabla \theta(x,t)|^3}{\| \theta_0 \|_{L^\infty}} + \frac{D(x,t)}{4} \leq |\nabla u (x,t)| \; |\nabla \theta(x,t)|^2. \label{eq:pointwise:ODE:1}
\end{align}

We estimate the absolute value of 
\begin{align*}
\nabla u(x,t) = \mathcal{R}^\perp \nabla \theta(x,t) =  P.V. \int_{\RR^d} \frac{(x-y)^\perp}{|x-y|^{d+1}} \left(\nabla \theta(x,t) -\nabla \theta(y,t) \right) dy
\end{align*}
by splitting {\em softly} (i.e. with smooth cut-offs) into an inner piece $|x-y|\leq \rho$, for some $\rho = \rho(x,t) >0$ to be chosen later, a medium piece  $\rho < |x-y| < L$ (where $L>0$ will be chosen later), and an outer piece $|x-y| \geq L$.
Bounding the inner piece follows directly from the Cauchy-Schwartz inequality:
\begin{align*}
|\nabla u_{in}(x,t)|\leq c_2 \sqrt{ D(x,t) \rho(x,t)}
\end{align*}
for some positive constant $c_2$ depending only on the smooth cut-off. We choose $\rho$ so that
\begin{align}
c_2 \sqrt{D \rho} |\nabla \theta|^2 \leq \frac{D}{8} + 2 c_2^2 \rho |\nabla \theta|^4  \leq \frac{D}{4} \label{eq:uin:bound}
\end{align}
which in view of \eqref{eq:pointwise:D:lower} is true if we let
\begin{align}
\rho(x,t) = \frac{c_1}{4 c_2^2 \|\theta_0\|_{L^\infty} |\nabla \theta(x,t)| } \label{eq:rho:def:1}.
\end{align}
Now we estimate $\nabla u_{med}$ using the $OSS_\delta$ property. Since $\int_{|x-y|=r} (x-y)^\perp |x-y|^{-d-1} \nabla \theta(x) dy = 0$ for $r>0$, and the cut-off used in the soft representation of the integral is radial, integrating by parts in $y$ we obtain
\begin{align}
|\nabla u_{med}(x,t)| \leq c_3 \int_{\rho \leq |x-y| \leq L} \frac{ |\theta(x,t) - \theta(y,t)|}{|x-y|^{d+1}} dy. \label{eq:umed:bound}
\end{align}
Using the assumption that $\theta(x,t)$ has the $OSS_{\delta_0}$ property with corresponding length $L$, we obtain 
\begin{align}
|\nabla u_{med}(x,t)| \leq c_4 \frac{\delta_0}{\rho} = c_5  \|\theta_0\|_{L^\infty} \delta_0 |\nabla \theta(x,t)|\label{eq:umed:bound:2}
\end{align}
by \eqref{eq:rho:def:1}, where $c_4 = c_3 |{\mathbb S}^{d-1}|$ and $c_5 = 4 c_4 c_2^2 / c_1$. In order to make sure that $|\nabla u_{med}(x,t)| |\nabla \theta(x,t)|^2$ does not exceed  half of the positive term in \eqref{eq:pointwise:ODE:1}, i.e. $c_1 |\nabla \theta(x,y)|^3/ |\theta_0\|_{L^\infty} $, we let 
\begin{align}
\delta_0  = \frac{c_1}{2 c_5 \| \theta_0 \|_{L^\infty}^2}. \label{eq:delta0:def}
\end{align}
Of course the value of $L$ corresponding to the above fixed $\delta_0$, might not be larger than the value of $\rho$ as defined in \eqref{eq:rho:def:1}, case in which we have $u_{med} = 0$. Lastly we bound $\nabla u_{out}$ similarly to \eqref{eq:umed:bound}, and obtain
\begin{align}
|\nabla u_{out}(x,t)| \leq 2 c_4 \frac{\|\theta_0\|_{L^\infty}}{L}. \label{eq:uout:bound}
\end{align}
Therefore, from \eqref{eq:pointwise:ODE:1}, \eqref{eq:uin:bound}, \eqref{eq:umed:bound:2}, \eqref{eq:delta0:def}, and \eqref{eq:uout:bound} we arrive at the pointwise inequality
\begin{align}
\frac{1}{2} \left( \partial_t + u \cdot \nabla + \Lambda \right) |\nabla \theta(x,t)|^2  +  \frac{c_1 |\nabla \theta(x,t)|^3}{2 \| \theta_0 \|_{L^\infty}}  \leq 2 c_4 \frac{\|\theta_0\|_{L^\infty} |\nabla \theta(x,t)|^2}{L} \label{eq:pointwise:ODE:2}
\end{align}
which gives 
$$
( \partial_t + u \cdot \nabla + \Lambda ) |\nabla \theta(x,t)|^2 \leq 0, \qquad  \mbox{whenever}\qquad |\nabla \theta(x,t)| \geq  \frac{4 c_4 \|\theta_0\|_{L^\infty}^2}{c_1 L} = C_{*}.
$$ 
Reading this at a maximum of $|\nabla \theta(x,t)|$ (if it exists) would at least formally conclude the proof of the Theorem. Indeed, at a point of maximum the gradient is $0$ and the fractional Laplacian is positive, implying  that whenever $\max |\nabla\theta (\cdot,t)|$ reaches $C_{*}$, its time derivative is negative and hence it can never exceed the threshold level $C_{*}$. 

In order to make the argument described here rigorous, one may proceed as follows. Let $\varphi(r) : [0,\infty) \to [0,\infty)$ be a non-decreasing $C^{2}$ convex function that vanishes identically for $0 \leq r\leq \max\{ \|\nabla \theta_{0}\|_{L^{\infty}}^{2}, C_{*}^{2}\}$, is strictly positive for $r > \max\{ \|\nabla \theta_{0}\|_{L^{\infty}}^{2}, C_{*}^{2}\}$ and grows algebraically at infinity. Due to the convexity of $\varphi$ as in~\cite{C1,CC} we have
\begin{align*}
\varphi'( |\nabla \theta(x,t)|^{2} ) \Lambda |\nabla \theta (x,t)|^{2} \geq \Lambda \varphi (|\nabla \theta(x,t)|^{2})
\end{align*}
pointwise in $x$. Thus, we may multiply \eqref{eq:pointwise:ODE:2} by $\varphi'(|\nabla \theta(x,t)|^{2})$ and obtain
\begin{align}
\frac{1}{2} (\partial_{t} + u \cdot \nabla + \Lambda) \varphi(|\nabla \theta(x,t)|^{2}) \leq \frac{c_{1}|\nabla \theta(x,t)|^{2}}{2 \|\theta_{0}\|_{L^{\infty}}}  \bigl (C_{*} - |\nabla \theta(x,t)| \bigr) \varphi'(|\nabla \theta(x,t)|^{2})  \leq 0
\label{eq:pointwise:ODE:3}
\end{align}
since $\varphi' (|\nabla \theta|^{2})= 0$ for $|\nabla \theta|< C_{*}$. In particular, it follows from \eqref{eq:pointwise:ODE:3} that $\varphi(|\nabla \theta|^{2})$ satisfies the weak maximum principle
\begin{align}
 \| \varphi(|\nabla \theta(\cdot,t)|^{2})\|_{L^{\infty}} \leq \| \varphi(|\nabla \theta_{0}|^{2}) \|_{L^{\infty}}.\label{eq:ODE:conclude}
\end{align}
The above is for instance obtained from $L^{2p}$ estimates on \eqref{eq:pointwise:ODE:3}, using that $\int f^{2p-1} \Lambda f \geq 0$ for smooth functions $f$, and sending $p \to \infty$.
To conclude, we note that by design $\varphi( |\nabla \theta_{0}(x)|^{2}) = 0$ a.e., and hence from \eqref{eq:ODE:conclude} we obtain $\varphi( |\nabla \theta(x,t)|^{2}) = 0$ a.e., or equivalently  $| \nabla \theta(x,t)| \leq \max\{ \|\nabla \theta_{0}\|_{L^{\infty}},C_{*}\}$ for a.e. $x$, and all $t\in [0,T]$.
\end{proof}

\begin{theorem}[\bf Stability of Only Small Shocks] \label{thm:step2}
Let $\delta_0>0$ and $T>0$ be arbitrary. If $\theta_0$ has the $OSS_{\delta_0/8}$ property, then a bounded weak solution $\theta$ of the critical SQG equation has the uniform $OSS_{\delta_0}$ property on $[0,T]$.
\end{theorem}
\begin{proof}[Proof of Theorem~\ref{thm:step2}]
We take now $\delta_h \theta (x,t) = \theta (x+h,t)-\theta(x,t)$. The equation 
obeyed by $\delta_h\theta$ is
\begin{align}
\left (\partial_t + u\cdot\nabla_x +(\delta_h u)\cdot\nabla_h +\Lambda \right)\delta_h\theta= 0
\label{eqd}
\end{align}
where
\begin{align}
(\delta_h u)(x) = u(x+h,t)-u(x,t) = P.V.\int_{\RR^d}\frac{(x-y)^{\perp}}{|x-y|^{d+1}}\delta_h\theta(y,t) dy.
\label{delthu}
\end{align}
This looks like it might have a maximum principle in both $x$ and $h$, but of course, there is no decay as $|h|\to\infty$, and the maximum is not small.
We take 
\begin{align}
\Phi (h) = e^{-\Psi(|h|)}\label{phi}.
\end{align}
The main properties of $\Phi$ are: smooth, radial, strictly positive, non-increasing, $\Phi(0) =1$, and \\
$\lim_{|h|\to \infty}\Phi(|h|) = 0$. We will need therefore $\Psi$ to be positive,  $\Psi'>0$, normalized $\Psi(0) =0$ and $\lim_{l\to\infty}\Psi(l)=\infty$. We shall construct the specific $\Psi$ at the end of the proof.

We multiply \eqref{eqd} by $(\delta_h\theta)\Phi(h)$ and  obtain
\begin{align}
\frac{1}{2}\left(\partial_t + u\cdot\nabla_x + (\delta_h u)\cdot\nabla_h + \Lambda\right) (\delta_h\theta(x,t)^2\Phi(h)) + \frac{1}{2}\Phi(h)D_h = \frac{1}{2}(\delta_h\theta(x,t))^2(\delta_h u)\cdot\nabla_h\Phi(h)
\label{base}
\end{align}
where
\begin{align}
D_h(x,t) = c_0 \int_{\RR^d}\frac{(\delta_h\theta(x,t)-\delta_h\theta(y,t))^2}{|x-y|^{d+1}}dy
\label{dh}
\end{align}
where $c_0$ is the normalizing constant of the integral defining $\Lambda$.
Let us denote by $v=v(x,t;h)$
\begin{align}
v= (\delta_h\theta(x,t))^2\Phi(h)\label{v}
\end{align}
and by $L$ the operator
\begin{align}
L= \partial_t +u\cdot\nabla_x + (\delta_h u)\cdot\nabla_h + \Lambda.
\label{L}
\end{align}
Let us note that
\begin{align}
|\nabla_h\Phi|\Phi(h)^{-1}\le \Psi'(|h|) \label{phineq}
\end{align}
and so, from \eqref{base} and \eqref{phineq} we have
\begin{align}
Lv + \Phi(h)D_h \leq \Psi' |\delta_h u| v \label{ineqone}.
\end{align}
Now we will assume that $\theta_0\in L^p(\RR^d)\cap L^{\infty}(\RR^d)$, for some $1\leq p<\infty$. Then we can bound $\delta_h u$ by splitting as usual in an inner part
\begin{align*}
\delta_h u_{in} (x,t) = P.V. \int_{|x-y|\leq \rho}\frac{(x-y)^{\perp}(\delta_h\theta(x,t)-\delta_h \theta(y,t))}{|x-y|^{d+1}}dy
\end{align*}
a medium part when $\rho<|x-y|\leq R$, and an outer part. 
We note that
\begin{align}
|\delta_h u_{in}(x,t)|\leq c_1 \sqrt{\rho D_h(x,t)}\label{dhuineq}
\end{align}
follows immediately from the Cauchy-Schwartz inequality.
For the medium part we have
\begin{align}
|\delta_h u_{med}(x,t)|\leq c_{\infty}\|\theta_0\|_{L^{\infty}}\log\left(\frac{R}{\rho}\right)\label{uhmed}
\end{align}
while for the outer part, we use integrability in $L^p$ of $\theta$, to obtain
\begin{align}
|\delta_h u_{out}(x,t)|\leq c_p\|\theta_0\|_{L^p}R^{-\frac{d}{p}}.
\label{dhoutineq}
\end{align}
We distinguish between the cases
$D_h\leq 1$ and $D_h>1$. If $D_h\leq 1$ we choose $\rho = R=1$ and
we obtain in this case
\begin{align*}
|\delta_h u(x,t)|\leq c_1+ c_p\|\theta_0\|_{L^p}.
\end{align*}
If $D_h>1$ we choose $\rho=D_h^{-1}$ and $R=1$ to obtain
\begin{align*}
|\delta_h u(x,t)| \leq c_1 + c_{\infty}\|\theta_0\|_{L^{\infty}}\log(D_h) +c_p\|\theta_0\|_{L^p}.
\end{align*}
Summarizing, we obtain
\begin{align}
|\delta_h u(x,t)| \leq  C_1 + C_{\infty}\log_{+}(D_h) 
\label{dhub}
\end{align}
where we set
\begin{align}
C_1 = c_1+ c_p\|\theta_0\|_{L^p}, \quad p<\infty
\label{C1}
\end{align}
and
\begin{align}
C_{\infty} = c_{\infty}\|\theta_0\|_{L^{\infty}}.
\label{Cinfty}
\end{align}

Now we use the elementary inequality
\begin{align}
Cb\log a \le \frac{a}{2} + Cb\log(2Cb) \label{eq:log:inequality}
\end{align}
valid for $C>0, a>0, b>0$. This inequality follows immediately from
\begin{align*}
\frac{a}{2Cb}<e^{\frac{a}{2Cb}}.
\end{align*}
In fact, \eqref{eq:log:inequality} also holds with $\log$ replaced by $\log_+$.
Applying this inequality with $C=\Psi' C_{\infty}$,\; $a=D_h$ and  $b=(\delta_h\theta(x,t))^2$, we obtain
\begin{align}
|\delta_hu(x,t)|\Psi' v \leq \Phi\frac{D_h}{2} + \Big( C_1+
C_{\infty}\log_{+}(2\Psi' C_{\infty}|\delta_h\theta(x,t)|^2)\Big) \Psi' v.
\label{bal}
\end{align} 
Using \eqref{bal} we obtain from \eqref{ineqone}
\begin{align}
Lv + \frac{1}{2}\Phi(h)D_h \leq  \Big( C_1+
C_{\infty}\log_{+}(8\Psi' c_{\infty}\|\theta_0\|_{L^{\infty}}^3)\Big)\Psi' v
\label{lvineq}
\end{align} 
Let us take now $r>0$ and observe that
\begin{align}
D_h(x,t)\geq c_4\frac{(\delta_h\theta(x))^2}{r} -2c_0 (\delta_h\theta(x))\int_{|x-y|\geq r}\frac{\delta_h\theta(y)}{|x-y|^{d+1}}\; \chi\left (\frac{x-y}{r}\right )dy\label{dhlowb}
\end{align}
with $\chi(y)$ a non-negative radial cut-off vanishing for $|y| <1$ and identically  equal to $1$ for $|y| \geq 2$. If $r\geq 3|h|$ we obtain 
\begin{align}
D_h(x,t)\geq c_5\frac{(\delta_h\theta(x,t))^2}{r} - c_6\frac{\|\theta_0\|_{L^{\infty}}|h|}{r^2}|\delta_h\theta(x,t)|
\label{condo}
\end{align}
Without loss of generality, we may assume $c_6\geq 3c_5$. Then we define
\begin{align}
r = \frac{2c_6|h|\|\theta_0\|_{L^{\infty}}}{c_5|\delta_h\theta(x,t)|}\label{rdefine}
\end{align}
In view of the fact that $|\delta_h\theta(x,t)|\leq 2\|\theta_0\|_{L^{\infty}}$,
we have that $r\geq 3|h|$, and from \eqref{condo} we deduce that
\begin{align}
\frac{1}{2}D_h(x,t) \geq c_7\frac{|\delta_h\theta(x,t)|^3}{|h|\|\theta_0\|_{L^{\infty}}}
\label{con}
\end{align}
and consequently
\begin{align}
\frac{1}{2}\Phi D_h(x,t) \geq c_7\frac{v^{\frac{3}{2}}}{|h|\|\theta_0\|_{L^{\infty}}}\Phi^{-\frac{1}{2}}
\label{condi}
\end{align}
Using the lower bound \eqref{condi} in \eqref{lvineq}, we deduce that
\begin{align}
Lv +\frac{c_7}{|h|\|\theta_0\|_{L^{\infty}}}v^{\frac{3}{2}}\Phi^{-\frac{1}{2}}
\leq C_{max}\left(1 + \log(1+ \Psi' )\right) \Psi' v  \label{lvineqn}
\end{align}
where 
\begin{align}
C_{max} = \max\{ C_p, C_{\infty}\}
\label{cmax}
\end{align}
and
\begin{align}
C_p =1 + c_p\|\theta_0\|_{L^p} + c_{\infty}\|\theta_0\|_{L^{\infty}}\log_+\left (8c_{\infty}\|\theta_0\|_{L^{\infty}}^3\right).\label{Cp}
\end{align}
Let us denote
\begin{align} 
q= \frac{\delta_0c_7\|\theta_0\|_{L^{\infty}}}{4C_{max}}.\label{q}
\end{align}
Note that $q$ is a fixed constant that depends in an explicit and  computable manner on $\delta_0$, $\|\theta_0\|_{L^{\infty}}$, and $\|\theta_0\|_{L^p}$ alone. 
Let us now assume that $\Psi$ satisfies
\begin{align}
\Psi'(y)(1+\log(1+\Psi'(y))) \le \frac{q}{y}\label{basipsi}
\end{align}
for all $y>0$. Then we deduce that
\begin{align}
Lv\leq 0\label{walla}
\end{align}
holds whenever  $v\geq \frac{\delta_0^2}{16}\Phi$ and, a forteriori (because $\Phi\leq 1$), whenever $v\geq \frac{\delta_0^2}{16}$.
Let us denote
\begin{align}
F(p) = p(1+\log(1+p))
\label{F}
\end{align}
defined for $p\geq 0$. Clearly
\begin{align*}
F'(p) = 1+ \log(1+p) +\frac{p}{1+p} \geq 1,
\end{align*}
so $F$ is strictly increasing and
\begin{align*}
F(p)\ge p
\end{align*}
because $F(0)=0$. The inequality we need is therefore
\begin{align}
\Psi'(y) \leq F^{-1}\left(\frac{q}{y}\right)
\label{inv}
\end{align}
where $F^{-1}$ is the inverse function.
Now, because $p\leq F$ it follows that 
\begin{align*}
1+\log(1+p)\leq 1+\log(1+F)
\end{align*}
and in view of 
\begin{align*}
p = \frac{F}{1+\log(1+p)}
\end{align*}
we have
\begin{align}
p\geq \frac{F}{1+\log(1+F)}\label{yF}.
\end{align}
Reading this at $p=F^{-1}\left(\frac{q}{y}\right)$, i.e., for $F(p) =\left(\frac{q}{y}\right)$ we have
\begin{align}
\frac{q}{y\left(1 + \log\left (1 +\frac{q}{y}\right)\right)} \leq F^{-1}\left(\frac{q}{y}\right)\label{wall}.
\end{align}
Therefore \eqref{inv} and thus \eqref{basipsi} will be satisfied if 
\begin{align}
\Psi'(y) \leq \frac{q}{y\left(1 + \log\left (1 +\frac{q}{y}\right)\right)}.
\label{ode}
\end{align}
The function
\begin{align*}
x\mapsto \frac{1}{x(1+\log(1+\frac{1}{x}))}
\end{align*}
is not integrable near infinity, nor near zero. This is a good thing.
The right hand side of \eqref{ode} tends to infinity as $y\to 0$.
Let us pick $l>0$ and define
\begin{align}
G_l(y) = \int_l^y\frac{q}{x\left(1+\log\left(1+\frac{q}{x}\right)\right)}dx \label{G}.
\end{align}
We define $\Psi(y)$ to equal identically $\Psi(y) = 0$ for $0\le y\leq \frac{l}{2}$, $\Psi(y) = G_l(y)$ for $y>l$ and satisfying \eqref{ode} for all $y$.
More precisely, we take a smooth function $\phi_l(y)$, equal to $0$ for $0\leq y\leq l/2$, equal to $1$ for $y\geq l$,
satisfying $0\leq \phi_l(y)\leq 1$ for all $y$  and 
define
\begin{align}
\Psi'(y) = \phi_l(y)\frac{q}{y\left(1 + \log\left (1 +\frac{q}{y}\right)\right)}\label{psiprimeformula}
\end{align}
and consequently
\begin{align}
\Psi(y) = \int_0^y\phi_l(y)\frac{q}{y\left(1 + \log\left (1 +\frac{q}{y}\right)\right)}dy\label{psiformula}.
\end{align}
Because $v\in L^1(\RR^{2d})\cap L^{\infty}(\RR^{2d})$ and $L$ has a weak maximum principle we have
\begin{align}
\|v(\cdot, t; \cdot)\|_{L^{\infty}} \le \|v_0(\cdot; \cdot)\|_{L^{\infty}}\label{vinfty}
\end{align}
for all $0\leq t\leq T$ if $\|v_0\|_{L^{\infty}}<\frac{\delta_0^2}{16}$. 
Now, if $\theta_0$ has the property $OSS_{\delta_0/8}$ with with length $ L_0$, and because 
\begin{align*}
\lim_{l\to0}G_l(y) = \infty,
\end{align*}
by choosing $l$ small enough (depending on $L_0$ and $q$) we can assure that
\begin{align*}
\|v_0\|_{L^{\infty}}< \frac{\delta_0^2}{16} 
\end{align*}
and so 
\begin{align}
|v(x,t;h)|< \frac{\delta_0^2}{16}
\label{viq}
\end{align}
for all $x$, $h$ and all $t\leq T$, which means that
\begin{align}
|\delta_h\theta(x,t)|\leq \frac{\delta_0}{4}e^{\frac{1}{2}\Psi(|h|)}\label{final}
\end{align}
holds for all $x$, all $h$ and all $t\leq T$. Therefore $|\delta_h \theta(x,t)| \leq \delta_0$ for all $x$, $t\leq T$, and $|h|\leq \Psi^{-1}(2 \log 4)$.

In order to rigorously justify the maximum-principle-type estimate \eqref{viq}, one may proceed as in the proof of Theorem~\ref{thm:step1}. Namely, we introduce a smooth non-decreasing convex function $\varphi(r)$ which is identically $0$ on $0\leq r \leq \delta_0^2/16$, and positive otherwise. Multiplying \eqref{lvineq} with $\varphi'(v)$, and using the specific choice of $\Psi$ made above we may in fact show that $L \varphi(v(x,t;h)) \leq 0$ pointwise in $x,h$, and $t$. Since $L$ has a weak maximum principle, we obtain \eqref{vinfty} with $v$ replaced by $\varphi(v)$. Our suitable choice of $l$ small enough ensures that $\varphi(v_0) = 0$ a.e., which then proves \eqref{viq}, concluding the proof of the theorem.
\end{proof}

Combining Theorems~\ref{thm:step1} and \ref{thm:step2} we arrive at:
\begin{theorem}[\bf Global well-posedness for critical SQG] \label{thm:SQG:critical}
Let $\theta_0 \in \Schwartz(\RR^d)$. Then there exits a unique global in time smooth solution $\theta(x,t)$ of the initial value problem associated with the critical SQG equation.
\end{theorem}
\begin{proof}[Proof of Theorem~\ref{thm:SQG:critical}]	
Fix $\delta_0$ as in the statement of Theorem~\ref{thm:step1}, and pick an arbitrary $T>0$. There exists a bounded weak solution $\theta(x,t)$ on $[0,T]$. Since $\theta_0$ is in particular uniformly continuous, it automatically has the $OSS_{\delta_0/8}$ property. By Theorem~\ref{thm:step2}, it follows that $\theta(x,t)$ has the uniform $OSS_{\delta_0}$ property on $[0,T]$, and hence by Theorem~\ref{thm:step1} we have that $\theta \in L^\infty([0,T];W^{1,\infty}(\RR^d))$. This is sub-critical information which may be used to bootstrap and show that $\theta \in C^\infty( (0,T) \times \RR^d)$. From the proof it is clear that in fact we only need the initial data to lie in $W^{1,\infty}(\RR^d)$ and have sufficient decay at spacial infinity.
\end{proof}

%\subsection{The super-critical case}\label{sec:SQG:supercritical}
%For the super-critical SQG equation, $\alpha<1$, it is known that if $\theta \in L^\infty(0,T;C^\delta)$ with $\delta > 1-\alpha$, then $\theta$ is a regular solution on $(0,T]$ (cf.~\cite{CW08}). Remaining on the H\"older scale it is also known that if $\theta \in C(0,T;C^{1-\alpha})$ then it is regular on $(0,T]$ (cf.~\cite{DongPavlovic}). Both of these results were proven using the Littlewood-Paley characterization of H\"older spaces. Here we show that the arguments used in the critical case may be suitably modified to prove that $\theta \in L^\infty(0,T;C^{1-\alpha})$ is a sufficient condition for the regularity of $\theta$ on $(0,T]$. To the best of our knowledge this appears to be new.\footnote{Peter: If I am not mistaken this indeed doesn't follow from \cite{CW08} nor from \cite{DongPavlovic}. Am I wrong?}
%\begin{theorem}[\bf Conditional regularity for the super-critical SQG]\label{thm:SQG:supercritical}
%Let $\theta_0 \in \Schwartz(\RR^d)$ and $\theta \in L^\infty(0,T;C^{1-\alpha})$ be a corresponding H\"older continous weak solution of the super-critical SQG equation \eqref{eq:SQG:1}--\eqref{eq:SQG:3}. Then $\theta$ is Lipschitz, and hence $C^\infty$ smooth on $[0,T]$.
%\end{theorem}
%\begin{proof}[Proof of Theorem~\ref{thm:SQG:supercritical}]
%It is clear that a suitable a priori control on $\|\nabla \theta \|_{L^\infty}$ concludes the proof of the theorem. For this purpose, similarly to the proof of Theorem~\ref{thm:step1} we write
%\end{proof}

\begin{remark}[\bf Conditional regularity for the super-critical SQG]
By combining the proof of Theorem~\ref{thm:lowerbound:scalar} and that of Theorem~\ref{thm:lowerbound:extended}, one may show that if $\theta \in L^\infty(0,T;C^\delta)$, then
\begin{align*}
\nabla \theta  (x,t) \cdot \Lambda^\alpha \nabla \theta(x,t) \geq \frac{1}{2}\Lambda^\alpha |\nabla \theta(x,t)|^2 + c_1 \frac{ |\nabla \theta(x,t)|^{2+\frac{\alpha}{1-\delta}} }{ M^{\frac{\alpha}{1-\delta}}}
\end{align*}
where $M = \|\theta\|_{L^\infty_t C^\delta_x}$. Therefore,
\begin{align} 
\frac{1}{2} ( \partial_t + u \cdot \nabla + \Lambda^{\alpha}) |\nabla \theta(x,t)|^2 + c_1\frac{ |\nabla \theta(x,t)|^{2+\frac{\alpha}{1-\delta}} }{ M^{\frac{\alpha}{1-\delta}}} + \frac{D(x,t)}{2} \leq |\nabla u(x,t)| \; |\nabla \theta(x,t)|^2 \label{eq:supercrit:1}
\end{align}
with $D$ as defined by \eqref{eq:DDD:def}. To bound $|\nabla u(x,t)|$ we split in two pieces, according to $\rho >0$. The inner piece is bounded by $c_2 D^{1/2} \rho^{\alpha/2}$, while the outer piece is bounded by $c_2 M/ \rho^{1-\delta}$. The Schwartz inequality and optimizing in $\rho$ gives
\begin{align} 
|\nabla \theta(x,t)|^2 |\nabla u(x,t)| \leq \frac{D(x,t)}{2} + c_4 M^{\frac{\alpha}{1-\delta+\alpha}} |\nabla \theta(x,t)|^{4 - \frac{2\alpha}{1-\delta+\alpha}}. \label{eq:supercrit:2}
\end{align}
Therefore, it follows from \eqref{eq:supercrit:1} and \eqref{eq:supercrit:2} that if
\begin{align*}
2 + \frac{\alpha}{1-\delta} > 4 - \frac{2\alpha}{1-\delta+\alpha} \Leftrightarrow \frac{\alpha^2}{(1-\delta)^2} + \frac{\alpha}{1-\delta} > 2 \Leftrightarrow \delta > 1-\alpha
\end{align*}
then the maximum of $|\nabla \theta|$ can not exceed a certain constant which depends on $M$, showing that $\nabla \theta \in L^\infty(0,T; L^\infty(\RR^d))$ and hence $\theta$ is a regular solution on $(0,T)$. This recovers the results of \cite{CW08}, without making use of the Besov-space techniques.
\end{remark}

\section{Critical Burgers in \texorpdfstring{$d$}{d} dimensions} 
\label{sec:Burgers}

Returning to the example from the Introduction, we can use the exact same strategy to prove global existence of smooth solutions for critical Burgers equations. We consider
$$
\pa_t\theta + \fr{1}{2}|\nabla\theta |^2 + \Lambda \theta =0 %\la{hj}
$$
differentiate and obtain
\bn
\pa_t u + u\cdot\nabla u  + \Lambda u = 0
\la{ub}
\end{align}
for $u =\nabla\theta$. It is easy to show that $\|u\|_{L^{\infty}}$ is non-increasing as long as solutions are smooth. We have:
\begin{theorem}[{\bf Global regularity for critical $d-D$ Burgers}]\label{thm:burgers}
Assume $u_0 \in \Schwartz(\RR^d)$. The Cauchy problem for the Burgers equation \eqref{ub} is globally well posed in the smooth category.
\end{theorem}
We only give a sketch of the proof, since almost all arguments can be carried over from the SQG case. Letting 
$$
g =|\nabla u|
$$
we have that
\bn
\left(\pa_t +u \cdot\nabla  +\Lambda \right) g^2 + D + 2g^3 =0
\la{gbu}
\end{align}
with 
$$
D(x,t)=c_0 P.V. \int_{\RR^d}\fr{\left |\nabla u(x)-\nabla u(y)\right |^2}{|x-y|^{d+1}}dy.
$$ 
We assume the property $OSS_{\delta_0}$ for $u$. The main difference between Burgers and SQG is that in \eqref{gbu} we have $g^3$ instead of $g^2 \RSZ u$, so that we to use the $OSS$ property in the lower bound for the positive term $D$, rather than in bounding $g^3$. 
We (softly) split the integral expression for $D$ according to  $\rho$ and $\delta_0$, and then optimize in $\rho$ to obtain
\bn
D(x,t) &\ge c_0 \int_{|x-y|\geq \rho} \fr{\left |\nabla u(x)-\nabla u(y)\right |^2}{|x-y|^{d+1}}dy
\notag\\
&\ge c_1 \fr{g^2}{\rho} - c_2 g \int_{\rho \leq |x-y| \leq L} \frac{|u(x)-u(y)|}{|x-y|^{d+2}} dy - 2 c_2 g \|u\|_{L^\infty} \int_{ |x-y| > L} \frac{1}{|x-y|^{d+2}} dy \notag\\
&\ge c_1 \fr{g^2}{\rho} - c_3 \fr{g \delta_0}{\rho^2} - c_3 \fr{ g \|u\|_{L^2}}{L^2} \notag\\
&\ge c_4\fr{g^3}{\delta_0} - c_4\fr{g\|u_0\|_{L^{\infty}}}{L^2}
\la{dbu}
\end{align}
where $L$ is the length scale of the property $OSS_{\delta_0}$. It is enough to have $\delta_0$ sufficiently small, e.g. less than $c_4/2$, to deduce that $g$ is bounded. This proves that the small shocks property implies regularity.

The proof of stability of $OSS$ is similar (and simpler) than the proof for SQG. There we had to work in order to bound $|\delta_h u|$ in terms of $|\delta_h \theta|$, but as above here there is no need to do this. Thus, the ODE inequality for $\Psi'$ (corresponding to \eqref{basipsi} above) does not involve logarithms, and is simply $\Psi'\le\fr{q}{y}$. This function is clearly also not integrable around $y=0$, and so the argument given in \eqref{F}--\eqref{final} can be carried through. We obtain thus global existence of smooth solutions for critical Burgers equations in $\RR^d$.

\section{A nonlocal anti-symmetric perturbation of the Euler equations} \label{sec:Euler}
The Euler equations are the classical model for the motion of an ideal incompressible fluid. These equations give rise to some of the most challenging problems in mathematical fluid dynamics.
See, for instance, the survey articles \cite{BT,C2,C3}, the books \cite{Ch,MB}, and references therein for a review of the subject.
When it comes to the issue of global existence and regularity of solutions to the Euler (and Navier-Stokes) equations, the problem is much better understood in the two-dimensional case.
The main reason is that in two dimensions the vorticity stretching term is absent, allowing one to obtain a global in time maximum principle for the vorticity. This maximum principle is the key ingredient in the proofs of global existence of smooth solutions to the Euler equations (coupled with the Br\'ezis-Gallou\"et or more precisely the Beale-Kato-Majda inequality, see for instance~\cite{BKM,KT}).

Here we would like to point out that  current methods for understanding the Euler equations, even in two dimensions, are not robust with respect to very mild perturbations in the equation. In this direction we consider the following example of a two-dimensional Euler equation with  solution dependent forcing that is {\em linear, nonlocal,} and {\em anti-symmetric}:
\begin{align}
&\partial_{t} u  + u\cdot \nabla u + \nabla p = A\; \RSZ_{1} u \label{eq:u1}\\
& \nabla \cdot u = 0\label{eq:u2}
\end{align}
where $\RSZ_1 = \partial_{1} \Lambda^{-1}$ is the Riesz transform with Fourier symbol $-i \xi_1/|\xi|$, and $A>0$ is some given amplitude of the perturbation. Here $(x,t) \in \RR^{2} \times [0,\infty)$. The immediate difficulty arising in the analysis of global smooth solutions to the initial value problem associated to \eqref{eq:u1}--\eqref{eq:u2} is the lack of an {a priori} control of the $L^{\infty}$ norm of the vorticity $\omega = \nabla^{\perp} \cdot u = \partial_{1} u_{2} - \partial_{2} u_{1}$, due to the unboundedness of the Riesz transform in $L^{\infty}$. Indeed, from the vorticity equation associated to \eqref{eq:u1}--\eqref{eq:u2}, i.e.
\begin{align}
&\partial_t \omega + u \cdot \nabla \omega = A\; \RSZ_1 \omega \label{eq:w1}\\
& \omega = \nabla^\perp \cdot u, \nabla \cdot u = 0 \label{eq:w2}
\end{align}
we only obtain global in time bounds on the  $L^{p}$ norms of $\omega$, with $1<p<\infty$. Alternatively, if one were able to control $\Vert \omega(\cdot,t) \Vert_{BMO}$, globally in time, the global regularity of \eqref{eq:u1}--\eqref{eq:u2} would also follow (cf.~\cite{KT}).
%However, we are not aware of techniques that enable one to estimate the growth in time of the $BMO$ norm of solutions to \eqref{eq:w1}--\eqref{eq:w2}.

We mention that recently the first author and collaborators have analyzed in \cite{CCW} the so-called {\em Loglog-Euler} equation,  i.e. the active scalar equation
\begin{align}
&\partial_{t} \theta + v \cdot \nabla \theta = 0 \label{eq:LogLogEuler:1}\\
&v = \nabla^{\perp} \Lambda^{-1} P(\Lambda) \theta \label{eq:LogLogEuler:2}
\end{align}
where $P(\Lambda)$ is a Fourier multiplier with symbol $$P(|\xi|) = \left( \log( 1+ \log(1 + |\xi|^{2})) \right)^{\gamma},$$ and $0 \leq \gamma \leq 1$.
If one regards $\theta$ as the vorticity, the difference between the system \eqref{eq:LogLogEuler:1}--\eqref{eq:LogLogEuler:2} and the classical Euler equations is that the  velocity $v$ is logarithmically more singular. To obtain the global regularity of smooth solution to \eqref{eq:LogLogEuler:1}--\eqref{eq:LogLogEuler:2}, with even more singular velocities, i.e. for functions $P$ that grow faster than $\log \log |\xi|$ as $|\xi| \rightarrow \infty$, is an open problem. We remark that the features which make  the systems \eqref{eq:w1}--\eqref{eq:w2} and \eqref{eq:LogLogEuler:1}--\eqref{eq:LogLogEuler:2} more difficult than the classical Euler equations are of different nature. For the Loglog-Euler system the $L^{\infty}$ maximum principle is available and the difficulty arrives from the borderline nature of the logarithmic estimate of $\Vert \nabla v \Vert_{L^{\infty}}$ in terms of $\Vert \theta \Vert_{L^{\infty}}$. For the system \eqref{eq:w1}--\eqref{eq:w2} estimating $\Vert \nabla u \Vert_{L^{\infty}}$ in terms of $\Vert \omega \Vert_{L^{\infty}}$ is done exactly the same as for the Euler equations, but we are lacking the a priori control on $\Vert  \omega \Vert_{L^\infty}$.

The motivation for addressing linear, nonlocal,  anti-symmetric perturbations of the Euler equations is quite basic: consider the solution-dependent forced equations
\begin{align}
&\partial_t u + u \cdot \nabla u + \nabla p = f(u),\  \nabla \cdot u = 0 \label{eq:forcedEuler}
\end{align}
where $f=(f_{1},f_{2})$, and $f_{1},f_{2} \colon \RR^{2} \rightarrow \RR$ are smooth and bounded functions. Such an equation may arise naturally for example if the Euler equations are coupled with another quantity that is transported by $u$, or for instance in the study of the stochastic Euler equations with multiplicative noise. In order to address the global in time regularity of \eqref{eq:forcedEuler}, one classically analyzes the equation solved by the vorticity $\omega = \nabla^{\perp} \cdot u$, namely
\begin{align}
\partial_{t} \omega + u \cdot \nabla \omega = - \partial_{1} f_{1} \omega - (\partial_{1} f_{2} + \partial_{2} f_{1}) \RSZ_{12} \omega+ (\partial_{2}f_{2} - \partial_{1}f_{1}) \RSZ_{11} \omega \label{eq:forcedEulerVorticity}
\end{align}
where $\RSZ_{ij}$ are iterated Riesz transforms $\partial_{i} \partial_{j} \Lambda^{-2} $, and we have used the two-dimensional Biot-Savart law $u = ( -\partial_{2} \Lambda^{-2} \omega, \partial_{1} \Lambda^{-2} \omega).$ While the first term on the  right side of of \eqref{eq:forcedEulerVorticity} is harmless for $L^{\infty}$ estimates on $\omega$, unless $f$ is such that $\partial_{1} f_{2} + \partial_{2} f_{1}  = \partial_{2} f_{2}  - \partial_{1} f_{1} = 0$ identically, the remaining two terms are both of the type $\nabla f \; \RSZ_{ij} \omega$, i.e. a bounded  smooth function multiplied by a Calder\'on-Zygmund operator acting on $\omega$. This prevents one from obtaining an $L^{\infty}$ maximum principle for $\omega$ using classical methods. Therefore it is natural to simplify the right side of \eqref{eq:forcedEulerVorticity}, and have it contain just a constant multiple of one Calder\'on-Zygmund operator acting on $\omega$, which for simplicity we take to be $\RSZ_{1} = \partial_{1} \Lambda^{-1}$ yielding \eqref{eq:w1}--\eqref{eq:w2}.

The principal result of this section is to prove that if one regularizes the system \eqref{eq:w1}--\eqref{eq:w2} by introducing a very { mildly} dissipative operator $\LL$, one may a priori obtain the global in time control of the $L^\infty$ norm of the vorticity, and hence the resulting equations are globally well-posed in the smooth category.  More precisely, we consider the  system
\begin{align}
&\partial_t \omega + \LL \omega +  u \cdot \nabla \omega = A \RSZ_1 \omega \label{eq1}\\
& \omega = \nabla^\perp \cdot u,\   \nabla \cdot u = 0,\label{eq2}
\end{align}
on $\RR^{2} \times [0,\infty)$, where $A>0$ is the amplitude of the perturbation and the dissipative operator $\LL$ is defined via
\begin{align}
\LL \omega(x) = P.V. \int_{\RR^{2}} \frac{\omega(x) - \omega(x-y)}{|y|^{2} m(|y|)} \; dy \label{eq:L}.
\end{align}
The smooth, non-decreasing function $m \colon [0,\infty) \rightarrow [0,\infty)$, is taken to satisfy
\begin{align}
&\int_{0}^{1} \frac{m(r)}{r} dr < \infty \label{eq:L:cond:1}
\end{align}
and for simplicity also assume that $m$ satisfies the doubling condition
\begin{align}
m(2r) < c\; m(r) \label{eq:L:cond:2}
\end{align}
for some universal constant $c>0$, and for all $r>0$.
The classical examples of an operators $\LL$ satisfying \eqref{eq:L}--\eqref{eq:L:cond:2} are the fractional powers of the Laplacian
\begin{align*}
\Lambda^\alpha \omega(x)  = c_{\alpha} P.V. \int_{\RR^{2}} \frac{\omega(x) - \omega(x-y)}{|y|^{2+\alpha}} \; dy 
\end{align*}
for $\alpha \in (0,1)$, so that $m(r) = r^{\alpha}/c_{\alpha}$, where $c_{\alpha}$ is a normalizing constant. However one may consider dissipative operators that are {\em weaker} than any power of the fractional Laplacian. For instance one may consider an operator $\LL$ defined via \eqref{eq:L}, with $m(r)$ an increasing positive function that behaves like $1/( - \ln r)^{1+\epsilon}$ for all sufficiently small $r$, and some $\epsilon>0$. Condition \eqref{eq:L:cond:1} says that at sufficiently small scales the dissipative operator $\LL$ is stronger than the forcing $A\, \RSZ_{1}$. 
%Condition \eqref{eq:L:cond:2} is  technical.  

The main result of this section is:
\begin{theorem}[\bf Global regularity]\label{thm:main}
Let $\LL$ be a dissipative operator defined by \eqref{eq:L}, with $m$ satisfying \eqref{eq:L:cond:1}--\eqref{eq:L:cond:2}, and let the initial data $\omega(\cdot,0) = \omega_{0}$ be smooth, e.g. in $H^{s}$, with $s>1$. Then the initial value problem associated to \eqref{eq1}--\eqref{eq2}  has a unique global in time solution $\omega \in C(0,\infty;H^{s})$.
\end{theorem}
In fact, the global regularity of smooth solutions to \eqref{eq1}--\eqref{eq2} still holds if the condition \eqref{eq:L:cond:1} is weakened to only assume that $\lim_{r\to 0+} m(r) = 0$ (cf.~Remark~\ref{rem:weak:m} below).  The proof of Theorem~\ref{thm:main} is based on classical Sobolev energy estimates, the Beale-Kato-Majda inequality (cf.~\cite{BKM}), and establishing the a priori control of the  $L^{\infty}$ norm of $\omega$. Obtaining a suitable bound on $\Vert \omega \Vert_{L^{\infty}}$ is the main difficulty, and in this direction we have the  following global in time estimate:
\begin{theorem}[\bf Global $L^\infty$ control]\label{thm:noblowup}
Let $\omega_{0} \in H^{s}$ for some $s>1$, and let $\LL$ be defined via \eqref{eq:L}--\eqref{eq:L:cond:2}. There exists a positive constant $M = M(A,m,\Vert \omega_{0} \Vert_{L^{\infty}}) $ such that if $\omega$ is a $H^{s}
$-smooth solution of the initial value problem associated to \eqref{eq1}--\eqref{eq2} on $[0,T)$, then  we have
\begin{align*}
\Vert \omega(\cdot,t) \Vert_{L^{\infty}} \leq   M
\end{align*}
for all $t\in [0,T)$.
\end{theorem}

Before we turn to the proof of Theorem~\ref{thm:noblowup}, we point out that the system \eqref{eq1}--\eqref{eq2} is conservative.

\begin{remark}[\bf Energy and enstrophy conservation] \label{rem:energy:enstrophy}
We note that $\int_{\RR^{2}} u \RSZ_{1} u \; dx = 0$ since $\RSZ_{1}$ is given by an odd Fourier symbol. Hence if we multiply the dissipative version of \eqref{eq:u1} by $u$ and integrate by parts we obtain that $\partial_{t} \int_{\RR^{2}} |u|^{2}(x,t)\; dx  = 0$,
 if $u$ is smooth enough. Therefore,  the energy $\Vert u \Vert_{L^{2}}^{2}$ is non-increasing for smooth solutions to \eqref{eq1}. Similarly, if  one multiplies \eqref{eq1} by $\omega$ and integrates by parts, one obtains that for smooth solutions the enstrophy $\Vert \omega \Vert_{L^{2}}^{2}$ is also non-increasing.
\end{remark}

\begin{proof}[Proof of Theorem~\ref{thm:noblowup}]
Multiplying \eqref{eq1} by $\omega(x)$ we obtain
\begin{align*}
 \frac{1}{2} ( \partial_{t} + u \cdot \nabla ) |\omega(x,t)|^{2} + \omega(x,t) \LL\omega(x,t) = A \omega(x,t) \RSZ_{1} \omega(x,t) % \label{eq:Euler:1}
\end{align*}
and using the pointwise identity (which may be proven the same as \eqref{eq:pointwise})
\begin{align*} 
\omega(x,t) \LL\omega(x,t) = \frac{1}{2} \LL ( |\omega(x,t)|^{2} ) + \frac{D(x,t)}{2}  %\label{eq:Euler:2}
\end{align*}
where as usual
\begin{align} 
D(x,t) = P.V. \int_{\RR^{2}} \frac{ (\omega(x,t) - \omega(x-y,t))^{2}}{|y|^{2} m(|y|)} dy \label{eq:Euler:3}
\end{align}
we obtain
\begin{align} 
\frac{1}{2} ( \partial_{t} + u \cdot \nabla + \LL) |\omega(x,t)|^{2} + \frac{D(x,t)}{2}  = A \omega(x,t) \RSZ_{1} \omega(x,t) \label{eq:Euler:4}.
\end{align}
In order to bound the right side of \eqref{eq:Euler:4} we split softly in the integral representation of the Riesz transform, and use the Cauchy-Schwartz inequality to obtain
\begin{align} 
A \omega(x,t) \RSZ_{1} \omega(x,t) &= c_{0} A \omega(x,t) P.V. \int_{\RR^{2}} ( \omega(x) - \omega(x-y) ) \frac{y_{1}}{|y|^{3}} dy \notag\\
&\leq c_0 A |\omega(x,t)| P.V. \int_{|y|\leq 1} \frac{ |\omega(x) - \omega(x-y)| }{|y| \sqrt{m(|y|)}} \cdot \frac{|y_1| \sqrt{m(|y|)}}{|y|^2} dy  \notag\\
& \qquad \qquad + c_1 A |\omega(x,t)| \int_{|y|>1} \frac{|\omega(x-y)|}{|y|^2} dy \notag\\
&\leq c_2 A |\omega(x,t)| \sqrt{D(x,t)} \left(\int_0^1 \frac{m(r)}{r} dr \right)^{1/2} + c_3 A |\omega(x,t)| \|\omega(\cdot,t)\|_{L^2} \notag\\
&\leq \frac{D(x,t)}{4} + c_4 A^2 |\omega(x,t)|^2+  c_3 A |\omega(x,t)| \|\omega_0\|_{L^2}.\label{eq:Euler:5}
\end{align}
Here $c_4$ depends linearly on $\int_0^1 \frac{m(r)}{r} dr$, a quantity assumed to be finite in \eqref{eq:L:cond:1}. Thus, \eqref{eq:Euler:4}--\eqref{eq:Euler:5} give
\begin{align} 
\frac{1}{2} ( \partial_{t} + u \cdot \nabla + \LL) |\omega(x,t)|^{2} + \frac{D(x,t)}{4} \leq c_4 A^2 |\omega(x,t)|^2+  c_3 A |\omega(x,t)| \|\omega_0\|_{L^2} \label{eq:Euler:6}
\end{align}
for a universal constant $c_3>0$, and a constant $c_4>0$ which may depend on $m$. To bound $D(x,t)$ from below, we let $\rho>0$ and since $m$ is increasing estimate
\begin{align} 
D(x,t) &\geq \int_{|y|\geq \rho}  \frac{ (\omega(x,t) - \omega(x-y,t))^{2}}{|y|^{2} m(|y|)} dy \notag\\
&\geq |\omega(x,t)|^2 \int_{\rho \leq |y| \leq 1} \frac{dy}{|y|^{2} m(|y|)}  - 2 |\omega(x,t)| \int_{|y|\geq \rho} \frac{ | \omega(x-y,t)|}{|y|^{2} m(|y|)} dy  \notag\\
&\geq \frac{c_5 }{m(1)}|\omega(x,t)|^2\ln \frac{1}{\rho} - c_6 |\omega(x,t)| \|\omega_0\|_{L^2}  \frac{1}{\rho m(\rho)} \label{eq:Euler:7}.
\end{align}
We could now optimize in $\rho$ and obtain a positive lower bound for $D$, but in fact there is no need to do that. We simply pick $\rho= \rho(A,m) \in (0,1)$ to be such that
$$
\frac{c_5}{8 m(1)} \ln \frac{1}{\rho} > c_4 A^2.
$$
For this fixed $\rho$ from \eqref{eq:Euler:6}--\eqref{eq:Euler:7} we obtain
\begin{align} 
\frac{1}{2} ( \partial_{t} + u \cdot \nabla + \LL) |\omega(x,t)|^{2} + c_6 |\omega(x,t)|^2 \leq  c_7 |\omega(x,t)| \|\omega_0\|_{L^2} \label{eq:Euler:8}
\end{align}
for some positive constants $c_6,c_7$ which may depend on $\rho$, $m$, and $A$. The a priori estimate \eqref{eq:Euler:8} shows that $(\partial_t + u \cdot \nabla + \LL) |\omega|^2 \leq 0$ whenever $|\omega(x,t)| \geq c_7 \|\omega_0\|_{L^2}/c_6$. Again, if the maximum of $|\omega(x)|$ were attained at some point $\bar x$, since $\nabla |\omega(\bar x)|^2 =0$, and $\LL \omega(\bar x) \geq 0$ we would formally obtain from \eqref{eq:Euler:8} that $\partial_t |\omega(\bar x)|^2 \leq 0$ whenever $|\omega(\bar x)|$ is too large, showing that the $L^\infty$ norm of $\omega$ can never exceed a certain value. As in the proof of Theorem~\ref{thm:step1},  in order to make this argument rigorous, we introduce a non-decreasing convex smooth function $\varphi(r)$ which is identically $0$  on $ 0 \leq r \leq \max\{ \|\omega_0\|_{L^\infty}^2 , c_7^2 \|\omega_0\|_{L^2}^2/ c_6^2 \}$, and positive otherwise.  Multiplying \eqref{eq:Euler:8} by $\varphi'(|\omega(x,t)|^2)$ then gives
\begin{align}
(\partial_t + u \cdot \nabla + \LL) \varphi( |\omega(x,t)|^2) \leq 0 \label{eq:Euler:9}
\end{align}
for all $x$ and all $t\in [0,T)$. It is not hard to verify (as in \cite{CC}) that $\LL$ is positive on $L^p$, i.e. 
$$
\int_{\RR^2} |f(x)|^{p-2} f(x) \LL f(x) dx \geq 0
$$
for all smooth functions $f$, and all $p$ even. Hence from \eqref{eq:Euler:9} we may obtain the weak maximum principle
$$
\| \varphi( |\omega(\cdot,t)|^2) \|_{L^\infty} \leq \| \varphi( |\omega_0|^2) \|_{L^\infty} = 0 
$$ 
due to our choice of $\varphi$. This shows that $\| \omega(\cdot,t)\| \leq M = \max\{ \| \omega_0\|_{L^\infty}, c_7 \|\omega_0\|_{L^2}/c_6\}$ for all $t \in [0,T)$, concluding the proof of the theorem.
\end{proof}

\begin{remark}[\bf Global well-posedness with arbitrarily weak dissipation]\label{rem:weak:m}
Using an argument inspired by the work \cite{KN1} for the critically dissipative dispersive SQG equation, one may obtain the global $L^{\infty}$ bound stated in Theorem~\ref{thm:noblowup} under {\em weaker} conditions on $m$, namely if \eqref{eq:L:cond:1} is replaced by 
\begin{align}
\lim_{r\to 0 +} m(r) = 0.\label{eq:L:cond:3}
\end{align}
The main ideas is as follows. If we were to assume that for some fixed time $t$, the $\sup_{x\in \RR^{2}} \omega(x,t)$ is attained at $\omega(\bar x)$, then at $\bar x$ the advective term in \eqref{eq1} vanishes and we are left with
\begin{align*}
\partial_{t} \omega(\bar x) = \int_{\RR^{2}} \frac{\omega(\bar x - y) - \omega(\bar x)}{|y|^{2}} \left( \frac{1}{m(|y|)} + \frac{A y_{1}}{2\pi |y|} \right) \; dy
\end{align*}
in the principal value sense. The smallness of $m(r)$ with respect to $1/A$ as $r\rightarrow 0+$ implies that there exists a small enough $\rho>0$ so that when restricted to $y\in B_{\rho}(0)$, the integral on the right side of \eqref{eq:L:cond:3} is negative. The contribution from the exterior of the ball $B_{\rho}$ is proportional to $\Vert \omega(\cdot,t) \Vert_{L^{2}} \leq \Vert \omega_{0} \Vert_{L^{2}}$. Therefore, at maxima of $\omega$ we have the bound
$$
\partial_t \omega(\bar x,t) \leq c \|\omega_0\|_{L^2}.
$$
A similar argument applied to the minimum would then show that $\| \omega(\cdot,t)\|_{L^\infty}$ may be  bounded as $\|\omega_0\|_{L^\infty} + c t \|\omega_0\|_{L^2}$ for all $t \in [0,T)$. However, since such a point $\bar x$ where the maximum (or minimum) is attained may not exist, we need to look at the evolution \eqref{eq1} when multiplied by a smooth cut-off function, leaving us to estimating lower order terms. We omit further details.
\end{remark}

\begin{proof}[Proof of Theorem~\ref{thm:main}]
The local existence of smooth solutions  $u \in L^{\infty}_{t}H^{s}_{x}$, with $s>2$,  for the velocity equations \eqref{eq:u1}--\eqref{eq:u2} follows straightforward from the energy inequality
\begin{align}
\frac{1}{2} \frac{d}{dt} \Vert u \Vert_{H^{s}}^{2} \leq c_s \Vert \nabla u \Vert_{L^{\infty}} \Vert u \Vert_{H^{s}}^{2} \label{eq:energy}
\end{align}
and the Sobolev embedding theorem. Estimate \eqref{eq:energy} is proven the same as for the Euler equations, since $\int_{\RR^{2}} \partial^{\alpha} \RSZ_{1} u \cdot \partial^{\alpha} u \; dx = 0$ for all $\alpha \in {\mathbb N}^{2}$, and hence the term $\RSZ_{1} u$ is absent in $L^{2}$-based estimates.

To obtain the global existence of smooth solutions to \eqref{eq:u1}--\eqref{eq:u2}, the standard procedure is to bound the term $\Vert \nabla u \Vert_{L^{\infty}}$ with $\Vert \omega \Vert_{L^{\infty}}$ and a Sobolev extrapolation inequality with logarithmic correction (see, for instance~\cite{BKM,KT}). 
%First, using the vorticity formulation \eqref{eq:w1}, it is straightforward to prove that $\Vert u(\cdot,t) \Vert_{W^{1,p}} < \infty$ for all $t>0$, and for all $2 \leq p < \infty$. Indeed, multiplying \eqref{eq:w1} by $\omega |\omega|^{p-2}$ yields \
%\begin{align}
%\partial_{t} \Vert \omega \Vert_{L^{p}} \leq \Vert \RSZ_{1} \omega \Vert_{L^{p}} \leq c p \Vert \omega \Vert_{L^{p}}\label{eq:Lp:energy}
%\end{align} and hence $\Vert \omega(\cdot,t) \Vert_{L^{p}} \leq \Vert \omega_{0} \Vert_{L^{p}} \exp(c t p)$, for all $t>0$ and $2\leq p<\infty$. Bounds on $u$ then follow by interpolation, using the energy conservation.  However, since the Riesz transforms are not bounded on $L^{\infty}$,  it seems not directly possible to prove a maximum-principle for the vorticity, and we only obtain  $W^{1,p}$ global a priori estimates, which are not sufficient to prove regularity.
Since the term $\RSZ_{1} u$ vanishes in $H^{s}$ energy estimates, it is not hard to check that the following blow-up criterion may be obtained directly from \eqref{eq:energy} and the Beale-Kato-Majda inequality:
$$
\mbox{If } \lim_{t \nearrow T} \int_{0}^{t} \Vert \omega(\cdot,s) \Vert_{L^{\infty}}\; ds < \infty, \mbox{ then the smooth solution may be continued past } T.
$$
In addition, one may prove using similar arguments to \cite{CC} that  $\LL$ is positive on in $L^p$ estimates with $p$ even. It hence directly follows that the above blow-up criterion, which is proven for the  non-dissipative equations, also hold for the dissipative system \eqref{eq1}--\eqref{eq2}, and hence the global regularity of smooth solutions holds due to the a priori bound on $\|\omega(\cdot,t)\|$ obtained in Theorem~\ref{thm:noblowup}.
\end{proof}

Lastly, we point out that one may obtain a direct time-independent a priori estimate on  $\Vert \nabla \omega \Vert_{L^{\infty}}$, for smooth solutions $\omega$ of \eqref{eq1}--\eqref{eq2}. For simplicity we give a sketch the proof that $\Vert \nabla \omega(\cdot,t) \Vert_{L^{\infty}}$ is bounded as $t \to \infty$ if the dissipative operator $\LL$ is a fractional power of the Laplacian.

\begin{proposition}[\bf Uniform $W^{1,\infty}$ control]\label{cor:Lip}
Let $\omega \in C((0,T);H^{s})$ be a solution of the initial value problem associated to \eqref{eq1}--\eqref{eq2}, with  $\LL =\Lambda^\alpha$ for some $\alpha \in (0,2)$. Then we have
\begin{align*}
\Vert \nabla \omega (t) \Vert_{L^{\infty}} \leq M
\end{align*}
for all $t \in [0,T)$, where  $M = M( A, \alpha, \Vert \omega_{0} \Vert_{L^{2}}, \Vert \omega_{0} \Vert_{W^{1,\infty}})$, is a suitable constant.
\end{proposition}

\begin{proof}[Proof of Proposition~\ref{cor:Lip}]
Taking the gradient of \eqref{eq1} and taking an inner product with $\nabla \omega$ we obtain the pointwise bound (we omit time dependence)
\begin{align}
\frac{1}{2} ( \partial_{t} + u \cdot \nabla  +  \Lambda^\alpha) |\nabla \omega(x)|^2 + \frac{D(x)}{2} = A \RSZ_1 \nabla \omega(x) \cdot \nabla \omega(x) + \nabla u(x) \colon \nabla \omega(x) \cdot \nabla \omega(x). \label{eq:cor:Lip:1}
\end{align}
where as before
\begin{align*}
D(x) = \frac{c_{2,\alpha}}{2} P.V. \int_{\RR^2} \frac{ |\nabla \omega(x) - \nabla \omega(x+y)|^2}{|y|^{2+\alpha}} dy.
\end{align*}
We use half of the dissipation $D$ to obtain the nonlinear lower bound as in Theorem~\ref{thm:lowerbound:scalar}, while the other half is used to dampen the effect of the dispersive forcing. Decomposing the singular integral which defines $\RSZ_{1}$ into an inner piece and an outer piece according to the parameter $\rho>0$, similarly to the proof of Theorem~\ref{thm:step1} we obtain
\begin{align*}
A | \nabla \omega(x) \cdot \RSZ_{1} \nabla \omega (x) | \leq c_{1}  | \nabla \omega (x)| \sqrt{D(x)} \rho^{\alpha/2} + c_{1}  |\nabla \omega (x)| \frac{\| \omega \|_{L^{\infty}}}{\rho}
\end{align*} 
for some constant $c_{1} = c_{1}(A)>0$. After applying the Cauchy-Schwartz inequality and optimizing in $\rho$, the above estimate gives
\begin{align} 
A | \nabla \omega(x) \cdot \RSZ_{1} \nabla \omega (x) | \leq \frac{D(x)}{2} + c_{2} |\nabla \omega(x)|^{\frac{2+\alpha}{1+\alpha}} \| \omega\|_{L^{\infty}}^{\frac{\alpha}{1+\alpha}}. \label{eq:cor:Lip:2}
\end{align}
In order to estimate $\|\nabla u\|_{L^{\infty}}$,  we recall the bound (cf.~\cite{KT})
\begin{align}
\Vert f \Vert_{L^{\infty}} \leq c_3 \left( 1+ \Vert f \Vert_{BMO} (1 + \log_{+} \Vert f \Vert_{W^{s,p}} ) \right) \label{eq:KT}
\end{align}
which holds for $s> 2/p$, $1<p<\infty$, and some sufficiently large constant $c_3$. Letting $p$ be sufficiently large and $s \in (0,1)$, one may interpolate $W^{s,p}$ between $L^{2}$ and $W^{1,\infty}$ and then apply \eqref{eq:KT} to $f = \nabla u$. Since $\nabla u$ is a matrix of Riesz transforms acting on $\omega$, we have $\Vert \nabla u \Vert_{BMO} \leq C \Vert \omega \Vert_{L^{\infty}}$ and therefore
\begin{align}
\| \nabla u \|_{L^\infty} \leq c_4 \left(1 + \Vert \omega  \Vert_{L^{\infty}} (1 + \log_{+} ( \Vert \omega \Vert_{L^{2} \cap L^{\infty}} + \|\nabla \omega\|_{L^{\infty}}  ) ) \right) \label{eq:cor:Lip:3}.
\end{align}

Assume there exists $\bar x$,  a point where the maximum of $|\nabla \omega|^2$ is attained (this can be made rigorous using cut-off functions or convex change of variables). Evaluating \eqref{eq:cor:Lip:1} at $\bar x$, using \eqref{eq:cor:Lip:2}, \eqref{eq:cor:Lip:3}, and the lower bound on $D(x)$ given by Theorem~\ref{thm:lowerbound:scalar}, we obtain the a priori estimate
\begin{align}
\partial_{t} | \nabla \omega (\bar x)|^2 + c_5 \frac{ \Vert \nabla \omega \Vert_{L^{\infty}}^{2+\alpha} }{M^{\alpha} } \leq c_4 M \left( 1 + \log_{+}( M + \Vert \nabla \omega \Vert_{L^{\infty}} )  \right) \Vert \nabla \omega \Vert_{L^{\infty}}^2  + c_2 M \| \nabla \omega\|_{L^\infty}^{\frac{2+\alpha}{1+\alpha}} \label{eq:P}
\end{align}
where $M = M (\|\omega_0\|_{L^2}, \|\omega\|_{L^\infty_t L^\infty_x}$), which we knows is finite  cf.~Theorem~\ref{thm:noblowup}. Lastly, since on the left side of \eqref{eq:P} we have $\Vert \nabla \omega \Vert_{L^{\infty}}^{2+\alpha}$, while on the right side of \eqref{eq:P} we have the slower growing quantities $\Vert\nabla \omega \Vert_{L^{\infty}}^2 \log_{+} \Vert \nabla \omega \Vert_{L^{\infty}}$ and $\| \nabla \omega\|_{L^\infty}^{\frac{2+\alpha}{1+\alpha}}$, we obtain that $\partial_t |\nabla \omega (\bar x)|^2 \leq 0$, whenever $|\nabla \omega (\bar x)| = \|\nabla \omega\|_{L^\infty}$ is larger than a certain constant. Therefore  $\|\nabla \omega(\cdot,t) \|_{L^\infty}$ can never exceed this constant, for $t>0$.
\end{proof}

\section{Applications to the  2D Boussinesq equation with mixed fractional dissipation} \label{sec:Boussinesq}

The two-dimensional Boussinesq equations with mixed fractional dissipation of order $\alpha$ and $\beta$, denoted in the following as \BB{\alpha}{\beta}, is given by
\begin{align} 
& \partial_{t} u + u \cdot \nabla u + \nabla p + \Lambda^\alpha u = \theta e_{2} \label{eq:B:1}\\
& \nabla \cdot u = 0 \label{eq:B:2}\\
&\partial_{t} \theta + u \cdot \nabla \theta + \Lambda^\beta \theta = 0 \label{eq:B:3}
\end{align}
where $e_{2} = (0,1)$, and $\alpha,\beta \in [0,2]$. We make the convention that by $\alpha=0$ we mean that there is no dissipation in \eqref{eq:B:1}, and similarly $\beta =0$ means that there is no dissipation in \eqref{eq:B:3}.

The global well-posedness of smooth solutions to \BB{0}{0} is an outstanding open problem in fluid dynamics. Partial results have been obtained only in the presence of dissipation. The well-posedness of \BB{0}{2} and \BB{2}{0} have been obtained in \cite{Chae,HL}, while the scaling critical cases \BB{1}{0} and \BB{0}{1} have been resolved in \cite{HKR1} and \cite{HKR2} respectively. We also point out that the case of partial anisotropic dissipation has been considered in several settings cf.~\cite{CaoWu,DP,LLT} and references therein.

In this section we consider \BB{\alpha}{\beta}, with {\em both} $\alpha,\beta \in (0,1)$, and using the nonlinear maximum principles proven in Section~\ref{sec:max}, we prove the global regularity of smooth solutions under a certain condition between the powers of the fluid and transport dissipations. To the best of our knowledge the case when {both} $\alpha$ and $\beta$ are less than $1$ has not been previously addressed.  Our main result is:
\begin{theorem}[\bf Global well-posedness for \BB{\alpha}{\beta}] \label{thm:Boussinesq}
 Assume that $\theta_0$ and $u_0$ are sufficiently smooth and that $\nabla \cdot u_0 = 0$.  If  $\beta >2/(2+\alpha)$, then the Cauchy problem for the \BB{\alpha}{\beta} equations \eqref{eq:B:1}--\eqref{eq:B:3} has a unique global in time smooth solution $(u,\theta)$.
\end{theorem}
For clarity of the presentation we only give here the main ideas of the proof. These ideas may be turned into a rigorous proof using the arguments described at the end of Theorem~\ref{thm:step1}.
\begin{proof}[Proof of Theorem~\ref{thm:Boussinesq}]
Let $\omega = \nabla^\perp \cdot u$ be the vorticity associated to the velocity $u$. The evolution \eqref{eq:B:1} may be classically written in terms of the vorticity as:
\begin{align} 
& \left( \partial_{t}  + u \cdot \nabla + \Lambda^\alpha \right) \omega = \partial_1 \theta  \label{eq:B:4}.
\end{align}
It follows from \eqref{eq:B:3}--\eqref{eq:B:4} and energy estimates that even in the absence of any dissipative terms, i.e. for \BB{0}{0}, the equations are well-posed in the smooth category up to time $T$ if $\int_{0}^{T} \|\nabla \theta (\cdot,t) \|_{L^\infty} dt < \infty$.

From \eqref{eq:B:3} and the pointwise identity \eqref{eq:pointwise} we obtain that the evolution of $|\nabla \theta|^2$ is given by
\begin{align} 
& \frac{1}{2} \left( \partial_{t} + u \cdot \nabla + \Lambda^\beta\right) |\nabla \theta|^2 + \frac{D}{2}  = \nabla u \colon \nabla \theta \cdot \nabla \theta  \label{eq:B:5},
\end{align}
where cf.~\eqref{eq:DDD:def} we have
\begin{align*}
D(x)  = c_\beta P.V. \int_{\RR^2} \frac{ |\nabla \theta(x) - \nabla \theta(y)|^2 }{|x-y|^{2+\beta}} dy \end{align*}
and $c_\beta>0$ is a normalizing constant. Using the nonlinear lower bound of Theorem~\ref{thm:lowerbound:scalar}, and the $L^\infty$ maximum principle for $\theta$, we obtain that 
\begin{align} 
D(x) \geq c_1 \frac{|\nabla \theta(x)|^{2+\beta}}{\|\theta_0\|_{L^\infty}^\beta} \label{eq:B:7}
\end{align}
for some constant $c_1>0$. Thus, if we evaluate \eqref{eq:B:5} at a point $x_\theta = x_\theta(t)$ where the maximum of $|\nabla \theta(x,t)|^2$ is attained, and denoting 
\begin{align*}
\Theta(t)  = \| \nabla \theta(\cdot,t)\|_{L^\infty}
\end{align*}
we {formally} obtain from \eqref{eq:B:7} that
\begin{align} 
\partial_t |\nabla \theta(x_\theta,t)|^2 + c_1 \frac{\Theta(t)^{2+\beta}}{\|\theta_0\|_{L^\infty}^{\beta}} \leq \Theta(t)^2 \|\nabla u \|_{L^\infty}. \label{eq:B:8}
\end{align}
To obtain a similar a priori estimate for 
\begin{align*}
\Omega(t) = \| \omega(\cdot,t)\|_{L^\infty}
\end{align*}
we first observe that from \eqref{eq:B:1}--\eqref{eq:B:3} we have the energy bound
\begin{align} 
\|u(\cdot,t)\|_{L^2} \leq \| u_0\|_{L^2} + t \|\theta_0\|_{L^2} =: K(t) \label{eq:B:9}
\end{align}
which suggests that we should use the $L^2$ version of Theorem~\ref{thm:lowerbound:extended}. More precisely, we first multiply \eqref{eq:B:4} by $\omega(x)$ and then evaluate the equation at a point $\bar x = \bar x(t)$ where $|\omega(x,t)|^2$ achieves its maximum. We formally obtain from Theorems~\ref{thm:lowerbound:extended} and \ref{thm:lowerbound:scalar} that
\begin{align} 
\partial_t |\omega(\bar x,t)|^2 + c_2 \frac{\Omega(t)^{2+\frac{\alpha}{2}}}{K(t)^{\frac{\alpha}{2}}} \leq \Theta(t) \Omega(t) \label{eq:B:10}
\end{align}
Lastly, in order to couple the a priori estimates \eqref{eq:B:8} and \eqref{eq:B:10} we need a bound on $\|\nabla u \|_{L^\infty}$ in terms of $\|\omega\|_{L^\infty}$. While such a direct bound is not possible, using an extrapolation of the endpoint Sobolev inequality with logarithmic correction, and $H^s$ energy estimates with $s>1$, after some computations we may obtain
\begin{align} 
\|\nabla u(\cdot,t)\|_{L^\infty} \leq C_{0}+ C_{0}\Omega(t) +C_{0} \Omega(t) \log_{+} \left(1 + \int_{0}^{t} (1+K(\tau) + \Omega(\tau) + \Theta(\tau))^{\Gamma} d\tau \right)  \label{eq:B:energy}
\end{align}
where $C_0 = C_0 (\|\omega_0\|_{L^\infty \cap H^s}, \| \theta_0 \|_{L^\infty \cap H^s})$, for some $s\geq 3$, and $\Gamma = \Gamma(\alpha,\beta)>0$. We conclude the proof of the theorem under the assumption that \eqref{eq:B:energy} holds, and then return afterwards to give the proof of \eqref{eq:B:energy}.

 Assume by contradiction that solutions blow-up at $T>0$, and are smooth for all $t\in [0,T)$. Note that for all $t\in [0,T]$ we have $K(t) \leq K(T)$, and hence \eqref{eq:B:8}, \eqref{eq:B:10}, \eqref{eq:B:energy}, may be summarized as follows
 \begin{align} 
 &\partial_t |\nabla \theta(x_\theta,t)|^2 + C_{1} \Theta(t)^{2+\beta} \notag\\
 & \qquad \qquad \qquad \leq C_{0} \Theta(t)^2 \left(1 + \Omega(t) + \Omega(t) \log_{+} \left(1 + \int_{0}^{t} (1+K(T) + \Omega(\tau) + \Theta(\tau))^{\Gamma} d\tau \right)  \right)\label{eq:B:13} \\
&\partial_t |\omega(\bar x,t)|^2 + C_{2}  \Omega(t)^{2+\frac{\alpha}{2}} \leq \Theta(t) \Omega(t) \label{eq:B:12}
\end{align}
where $\Theta(t) = |\nabla \theta(x_\theta,t)|$ and $\Omega(t) = |\omega(\bar x,t)|$. The constants $C_{0},C_{1},C_{2}>0$ depend on various norms of the initial data, $\alpha, \beta$, and $K(T)$.
 
Let $M >0$ be a sufficiently large constant to be chosen precisely later. 
Assuming solutions blow-up at $T$, we must have that $\Theta(t)$ becomes unbounded (at least along a subsequence) as $t \rightarrow T$. Since $\Theta(t)$ is continous in time on $[0,T)$, we may define $T_{M} \in(0,T)$ to be the first time such that $\Theta(t) = M \geq 2 \Theta(0)$, that is, $[0,T_{M}]$ is the maximal time interval on which $\Theta(t) \leq M$ holds. 

Due to \eqref{eq:B:12}, on $[0,T_{M}]$ we must have
\begin{align*}
\Omega(t) \leq \max\left\{ \Omega(0) , \left( \frac{M}{C_{2}} \right)^{\frac{2}{2+\alpha}} \right\} =  \left( \frac{M}{C_{2}} \right)^{\frac{2}{2+\alpha}} =: \tilde{M}
\end{align*}
if $M$ is chosen sufficiently large in terms of $\Omega(0)$, $\alpha$, and $C_{2}$. The idea behind this is that whenever $\Omega(t) \geq \tilde{M}$, \eqref{eq:B:energy} shows that at the points $\bar x$ where $\Omega(t)$ is attained, we have $\partial_t |\omega(\bar x)|^2 \leq 0$, and so $\Omega(t)$ cannot exceed the value $\tilde{M}$. Feeding the above bound back into \eqref{eq:B:13}, we obtain that in fact on $[0,T_{M}]$ we must have
\begin{align*} 
 &\partial_t |\nabla \theta (x_\theta,t)|^2 + C_{1} \Theta(t)^{2+\beta} \leq C_{0} \Theta(t)^2 \left( 1+ \tilde{M} + \tilde{M} \log_{+} \left(1 + T (1+ K(T) \tilde{M} + M)^{\Gamma} \right) \right)
\end{align*}
at the points $x_\theta$ where $\Theta(t)$ is attained, and hence using a similar argument we get
\begin{align} 
\Theta(t)^{\beta} \leq \max\left\{ \Theta(0)^{\beta} , \frac{C_{0}}{C_{1}} \left( 1+ \tilde{M} + \tilde{M} \log_{+} \left(1 + T (1+ K(T) \tilde{M} + M)^{\Gamma} \right) \right)\right\}.\label{eq:B:14}
\end{align}
To conclude the proof we claim that if $M$ is chosen sufficiently large, the right side of \eqref{eq:B:14} can be made in fact less than $(M/2)^{\beta}$, which would then contradict the maximality of $T_{M}$. Recall that $\tilde{M} \approx M^{\frac{2}{2+\alpha}}$, and hence  up to constants the right side of \eqref{eq:B:14} is equal to $1 + M^{\frac{2}{2+\alpha}} (1 + \log_{+} M)$. Since $2/(2+\alpha) < \beta$, for any positive constant $C_{\infty}>0$, by letting $M$ be sufficiently large  we may ensure that 
\begin{align*}
1 + M^{\frac{2}{2+\alpha}} (1 + \log_{+} M)  \leq \frac{M^{\beta}}{C_{\infty}}.
\end{align*}
 It is clear that $M$ may be chosen to only depend on $T, K(T), \Omega(0), \Theta(0), C_{0},C_{1},C_{2}, \alpha,\beta$, i.e. on norms of the initial data, the candidate blow-up time, and on universal constants. This is the only place in the proof where we use the main assumption of the theorem, namely $\beta > 2/(2+\alpha)$.
 
 Thus we have proven that on $[0,T)$ $\Theta(t)$ can never exceed a finite constant $M>0$ (which may be computed in terms of the initial data and $T$), and so $\Theta(t)$ cannot blow-up as $t\rightarrow T$ whenever $T<\infty$, concluding the proof of the theorem.
\end{proof}

\begin{proof}[Proof of  \eqref{eq:B:energy}]
As in \eqref{eq:cor:Lip:3}, we first use a version of the Beale-Kato-Majda inequality as proven in~\cite[Theorem 1]{KT} and obtain
\begin{align} 
\|\nabla u\|_{L^\infty} \leq c \left(1 + \|\omega\|_{L^\infty} (1 + \log_+ \|\omega\|_{H^s} ) \right) \label{eq:app:1}
\end{align}
for some positive universal constant $c>0$, where $s>1$. Therefore, in order to prove \eqref{eq:B:energy} it is sufficient to find a bound on $\omega$ in $H^s$, for some $s>1$, which depends {\em polynomially} on $\|\omega\|_{L^\infty}$, and on the a priori controlled quantities $\| u\|_{L^2}$ and $\| \theta\|_{L^2 \cap L^\infty}$. 

Applying $\Delta$ to \eqref{eq:B:4}, multiplying with $\Delta \omega$ and integrating over $\RR^2$ we obtain
\begin{align}  
\frac{1}{2} \frac{d}{dt} \|\Delta \omega \|_{L^2}^2 + \frac{1}{4} \| \Lambda^{2+\alpha/2} \omega\|_{L^2}^{2} \leq \frac{1}{3} \| \Lambda^{2-\alpha/2} \partial_1 \theta \|_{L^2}^2 + T_1 + T_2 \label{eq:app:2}
\end{align}
where
\begin{align*}
 T_1 = 2 \left| \int \Delta u_j \partial_j \omega \Delta \omega \right| \qquad \mbox{and} \qquad 
 T_2 = 4 \left| \int \partial_k u_j  \partial_{jk} \omega \Delta \omega \right|
\end{align*}
and we have used the summation convention on repeated indices. Here we have used the estimate $ab \leq 3a^2/4 + b^2/3$. Upon integrating twice by parts, and using the H\"older inequality, we obtain
\begin{align}
T_1 \leq 2 \| \nabla \Delta u \|_{L^2} \| \Delta \omega \|_{L^2} \|\omega\|_{L^\infty} &\leq c  \| \Delta \omega \|_{L^2}^2 \|\omega\|_{L^\infty} \leq c  \| u\|_{L^2}^{\frac{2\alpha}{6+\alpha}} \| \Lambda^{2+\alpha/2} \omega \|_{L^2}^{\frac{12}{6+\alpha}} \|\omega\|_{L^\infty}. \label{eq:app:3}
\end{align}
For the last bound of \eqref{eq:app:3} we have used the Gagliardo-Nirenberg inequality. The bound on $T_2$ is a bit more involved. Let $p = (4-\alpha)/(2-\alpha)$. Then $p \in (2,4)$, and also $p< 4/(2-\alpha)$ since $\alpha < 1$. The H\"older, Calder\'on-Zygmund, and the Gagliardo-Nirenberg inequalities give
\begin{align} 
T_2 \leq c \| \nabla u \|_{L^{\frac{p}{p-2}}} \| \Delta \omega\|_{L^p}^{2}  
&\leq c \| \omega \|_{L^{\frac{p}{p-2}}} \| \Delta \omega\|_{L^p}^{2} \notag\\
&\leq c \left( \| \omega\|_{L^2}^{\frac{2\alpha}{4-\alpha}} \|\omega\|_{L^\infty}^{\frac{4-3\alpha}{4-\alpha}}\right) \left(\|u\|_{L^2}^{\frac{\alpha(2-\alpha)}{(4-\alpha)(6+\alpha)}}  \| \Lambda^{2+\alpha/2} \omega\|_{L^2}^{\frac{4(6-\alpha)}{(4-\alpha)(6+\alpha)}} \right)^{2} \notag\\
%& \leq \|u\|_{L^2}^{\frac{2\alpha(2-\alpha)}{(4-\alpha)(6+\alpha)} + \frac{2\alpha(4+\alpha)}{(4-\alpha) (6+\alpha)}}\| \Lambda^{2+\alpha/2} \omega\|_{L^2}^{\frac{8(6-\alpha)}{(4-\alpha)(6+\alpha)} + \frac{4\alpha}{(4-\alpha)(6+\alpha)}} \|\omega\|_{L^\infty}^{\frac{4-3\alpha}{4-\alpha}} \notag\\
&\leq c \|u\|_{L^2}^{\gamma_1(\alpha)}\| \Lambda^{2+\alpha/2} \omega\|_{L^2}^{\gamma_2(\alpha)} \|\omega\|_{L^\infty}^{\gamma_3(\alpha)} \label{eq:app:4}
\end{align}
where $\gamma_1(\alpha) = \frac{12 \alpha}{(4-\alpha)(6+\alpha)}$, $\gamma_2(\alpha) = \frac{48 -4\alpha}{(4-\alpha)(6+\alpha)}$, and $\gamma_3(\alpha) = \frac{24 - 14 \alpha - 3\alpha^2}{(4-\alpha)(6+\alpha)}$. While the specific values of $\gamma_1, \gamma_2$, and $\gamma_3$ are highly irrelevant, what is important is that $\gamma_2 < 2$, and $\gamma_1 + \gamma_2 + \gamma_3 = 3$. We obtain from the estimate \eqref{eq:app:2}, the bounds \eqref{eq:app:3}--\eqref{eq:app:4}, and the $\varepsilon$-Young inequality that
\begin{align} 
\frac{d}{dt} \|\Delta \omega \|_{L^2}^2  +\frac{1}{4} \| \Lambda^{2+\alpha/2} \omega\|_{L^2}^{2} \leq \frac{2}{3} \| \Lambda^{3-\alpha/2} \theta \|_{L^2}^2  + c \|u \|_{L^2}^2 \|\omega\|_{L^\infty}^{\frac{6+\alpha}{6}} + c \| u \|_{L^2}^{\frac{2\gamma_1(\alpha)}{2 - \gamma_2(\alpha)}} \|\omega \|_{L^\infty}^{\frac{2\gamma_3(\alpha)}{2 - \gamma_2(\alpha)}}, \label{eq:app:5}
\end{align}
where $c = c(\alpha)$ is a positive constant. Thus is it left to obtain a bound on $\| \Lambda^{3-\alpha/2} \theta\|_{L^2}^{2} = \| \Lambda^{r + \beta/2} \theta\|_{L^2}^{2}$, where we let $r = 3-\alpha/2 - \beta/2 >2$. Applying $\Lambda^r$ to \eqref{eq:B:3} and taking an $L^2$ inner product with $\Lambda^r \theta$, we obtain
\begin{align} 
\frac{1}{2} \frac{d}{dt} \| \Lambda^r \theta\|_{L^2}^{2} + \| \Lambda^{r +\beta/2} \theta\|_{L^2}^{2} &\leq \left| \int_{\RR^2} [ \Lambda^r, u \cdot \nabla ] \theta\; \Lambda^r \theta \right| \label{eq:app:6}
\end{align}
where $[\cdot,\cdot]$ denotes a commutator. We use the classical commutator estimate
\begin{align*} 
\|  [ \Lambda^r, u \cdot \nabla ] \theta \|_{L^2} \leq C \left( \| \nabla u\|_{L^{\infty}} \| \Lambda^r \theta\|_{L^2} + \| \Lambda^r u\|_{L^2} \|\nabla \theta \|_{L^{\infty}} \right).
\end{align*}
We have 
\begin{align} 
\| \Lambda^r \theta\|_{L^2}^{2} \| \nabla u \|_{L^\infty} &\leq \|\Lambda^{r+\beta/2} \theta\|_{L^2}^{\frac{4r}{2r +\beta} } \|\theta\|_{L^2}^{\frac{2\beta}{2r +\beta} } \| \nabla u \|_{L^\infty} \notag\\
&\leq \frac{1}{24}  \|\Lambda^{r+\beta/2} \theta\|_{L^2}^2 + c \|\theta_0\|_{L^2}^{\frac{\beta}{r}}  \| \nabla u \|_{L^\infty}^{\frac{2r+\beta}{2r}} \notag\\
&\leq  \frac{1}{24}  \|\Lambda^{r+\beta/2} \theta\|_{L^2}^2 + C_1( \|\theta_0\|_{L^2}, \alpha,\beta) \|\nabla u\|_{L^\infty}^{\delta_{1}(\alpha,\beta)} \label{eq:app:7}.
\end{align}
where $\delta_{1}(\alpha,\beta) = \frac{6-\alpha}{6-\alpha-\beta}$.
For the second term in the commutator we may bound 
\begin{align} 
\| \Lambda^r \theta \|_{L^2} \| \Lambda^r u\|_{L^2} \|\nabla \theta \|_{L^{\infty}} 
&\leq \|\Lambda^{r-1} \omega\|_{L^{2}} \|\nabla \theta\|_{L^{\infty}} \| \theta \|_{L^{2}}^{\frac{\beta}{2r+\beta}} \| \Lambda^{r+\beta/2} \theta\|_{L^{2}}^{\frac{2r}{2r+\beta}}\notag\\
&\leq \frac{1}{24} \| \Lambda^{r+\beta/2} \theta\|_{L^2}^{2} + C_2 ( \|\theta_0\|_{L^2}, \alpha,\beta)  \|\Lambda^{2-\frac{\alpha+\beta}{2}} \omega\|_{L^2}^{\delta_{2}(\alpha,\beta)} \|\nabla \theta\|_{L^{\infty}}^{\delta_{2}(\alpha,\beta)}\label{eq:app:8}
\end{align}
where $\delta_2(\alpha,\beta) = \frac{2(6-\alpha-\beta)}{6-\alpha+\beta}$. From \eqref{eq:app:6}, \eqref{eq:app:7}, and \eqref{eq:app:8} we thus obtain
\begin{align*} 
\frac{d}{dt} \| \Lambda^r \theta\|_{L^2}^{2} + \frac{1}{6} \| \Lambda^{r +\beta/2} \theta\|_{L^2}^{2} \leq C_1 \|\nabla u\|_{L^\infty}^{\delta_1(\alpha,\beta)} + C_2 \|\Lambda^{2-\frac{\alpha+\beta}{2}} \omega\|_{L^2}^{\delta_2(\alpha,\beta)}\|\nabla \theta\|_{L^{\infty}}^{\delta_{2}(\alpha,\beta)}
\end{align*}
and summing with \eqref{eq:app:5} we obtain
\begin{align} 
& \frac{d}{dt} \|\Delta \omega \|_{L^2}^2 +\frac{d}{dt} \| \Lambda^r \theta\|_{L^2}^{2}  +\frac{1}{4} \| \Lambda^{2+\alpha/2} \omega\|_{L^2}^{2} + \frac{1}{3} \| \Lambda^{r +\beta/2} \theta\|_{L^2}^{2} \notag\\
& \qquad   \leq c \|u \|_{L^2}^2 \|\omega\|_{L^\infty}^{\frac{6+\alpha}{6}}+  c \| u \|_{L^2}^{\frac{2\gamma_1(\alpha)}{2 - \gamma_2(\alpha)}} \|\omega \|_{L^\infty}^{\frac{2\gamma_3(\alpha)}{2 - \gamma_2(\alpha)}}  + C_1 \|\nabla u\|_{L^\infty}^{\delta_1(\alpha,\beta)} + C_2 \|\Lambda^{2-\frac{\alpha+\beta}{2}} \omega\|_{L^2}^{\delta_2(\alpha,\beta)} \|\nabla \theta\|_{L^{\infty}}^{\delta_{2}(\alpha,\beta)}\label{eq:app:9}
\end{align}
with $\gamma_1, \gamma_2,\gamma_3,\delta_1$ and $\delta_2$ as defined above.
To conclude, we bound
\begin{align} 
\|\nabla u \|_{L^\infty}^{\delta_1(\alpha,\beta)} \leq \| u \|_{L^2}^{\delta_1(\alpha,\beta) \frac{2+\alpha}{4+\alpha}} \| \Lambda^{2+\alpha/2} \omega\|_{L^2}^{\delta_1(\alpha,\beta) \frac{4}{6+\alpha}} \leq \frac{1}{8} \| \Lambda^{2+\alpha/2} \omega\|_{L^2}^2 + C \| u \|_{L^2}^{\gamma_4(\alpha,\beta)} \label{eq:app:10}
\end{align}
for some $\gamma_4(\alpha,\beta) >0 $. Here we used that 
$$
\delta_1(\alpha,\beta) \frac{4}{6+\alpha} = \frac{4 (6-\alpha)}{(6+\alpha)(6-\alpha-\beta)}<2.
$$ 
Also, we have
\begin{align} \label{eq:app:11}
\|\Lambda^{2-\frac{\alpha+\beta}{2}} \omega\|_{L^2}^{\delta_2(\alpha,\beta)}  \|\nabla \theta\|_{L^{\infty}}^{\delta_{2}(\alpha,\beta)} 
&\leq \| u\|_{L^2}^{\delta_2(\alpha,\beta) \frac{2 \alpha + \beta}{6+\alpha}} \| \Lambda^{2+\alpha/2} \omega\|_{L^2}^{\delta_2(\alpha,\beta)  \frac{6-\alpha-\beta}{6+\alpha}}  \|\nabla \theta\|_{L^{\infty}}^{\delta_{2}(\alpha,\beta)} \notag\\
&\leq \frac{1}{16}  \| \Lambda^{2+\alpha/2} \omega\|_{L^2}^2 + c \|u \|_{L^2}^{\gamma_5(\alpha,\beta)} \|\nabla \theta\|_{L^{\infty}}^{\gamma_{6}(\alpha,\beta)}
\end{align}
for some $\gamma_5(\alpha,\beta)$ and $\gamma_{6}(\alpha,\beta)> 0$. In the above we have used 
$$
\delta_2(\alpha,\beta) \frac{6-\alpha-\beta}{6+\alpha} = \frac{2(6-\alpha-\beta)^{2}}{(6+\alpha)(6-\alpha+\beta)} <2.
$$ 
Hence, inserting \eqref{eq:app:10}--\eqref{eq:app:11} into \eqref{eq:app:9}, we obtain the a priori estimate
\begin{align} 
& \frac{d}{dt} \left( \|\Delta \omega \|_{L^2}^2 + \| \Lambda^r \theta\|_{L^2}^{2}\right)  +\frac{1}{16} \| \Lambda^{2+\alpha/2} \omega\|_{L^2}^{2} + \frac{1}{3} \| \Lambda^{r +\beta/2} \theta\|_{L^2}^{2} \notag\\
& \qquad \qquad \leq C_{3} \left(1 + \|u \|_{L^2} + \|\omega\|_{L^\infty}+   \|\nabla \theta\|_{L^{\infty}}\right)^{\Gamma(\alpha,\beta)} \label{eq:app:12}
\end{align}
for some $C_{3}>0$ which may depend on various norms of the initial data, and for some possibly very large $\Gamma(\alpha,\beta)>0$. Integrating \eqref{eq:app:12} in time, omitting the positive dissipative terms, inserting into \eqref{eq:app:1}, gives
\begin{align*} 
\| \nabla u(t) \|_{L^{\infty}} &\leq C_{4} +C_{4}  \| \omega(t)\|_{L^{\infty}}\notag\\
& \qquad + C_{4} \| \omega(t)\|_{L^{\infty}} \log_{+} \left(1 + \int_{0}^{t}  \left(1 + \|u(\tau) \|_{L^2} + \|\omega(\tau)\|_{L^\infty}+   \|\nabla \theta(\tau)\|_{L^{\infty}}\right)^{\Gamma(\alpha,\beta)} d\tau \right)
\end{align*}
for some sufficiently large $C_{4}$ depending on norms of the initial data, concluding the proof of the a priori estimate \eqref{eq:B:energy}.
\end{proof}

\subsection*{Acknowledgments}PC's research was partially supported by NSF-DMS grant 0804380.

\end{document}